\documentclass[12pt,reqno,oneside]{amsart}

\usepackage{graphicx,subfigure}
\usepackage{amssymb}
\usepackage{epstopdf}

\usepackage[a4paper, total={6in, 9.1in}]{geometry}

\usepackage{amsmath,amsfonts,amsthm,mathrsfs,amssymb,cite}
\usepackage[usenames]{color}

\usepackage{graphics}
\usepackage{framed}

\usepackage{enumerate}

\usepackage{hyperref}
\usepackage{xcolor}
\hypersetup{
  colorlinks,
  linkcolor={black!80!black},
  urlcolor={blue!80!black},
  citecolor={blue!80!black}
}
\usepackage[capitalize,noabbrev]{cleveref}

\numberwithin{equation}{section}

\newtheorem{theorem}{Theorem}[section]
\newtheorem{lemma}[theorem]{Lemma}
\newtheorem{proposition}[theorem]{Proposition}

\newtheorem{corollary}[theorem]{Corollary}

\newtheorem{remark}[theorem]{Remark}
\numberwithin{equation}{section}

\def\bbr{{\mathbb R}}
\def\bbn{{\mathbb{N}}}
\newcommand{\mi}{\mathrm{i}}
\newcommand{\mb}[1]{\mathbf{#1}}
\newcommand{\coma}{\; , \;}
\newcommand{\dx}[1]{\, \mathrm{d} #1}

\title[Asymptotic acoustic behaviors due to high contrast]{On asymptotic behaviors of acoustic waves due to high-contrast material inclusions}

\author{Yueguang Hu}
\address{Department of Mathematics, City University of Hong Kong, Kowloon, Hong Kong, China}
\email{yueghu2-c@my.cityu.edu.hk}

\author{Hongyu Liu}
\address{Department of Mathematics, City University of Hong Kong, Kowloon, Hong Kong, China}
\email{hongyu.liuip@gmail.com, hongyliu@cityu.edu.hk}

\makeatletter
\begin{document}

\begin{abstract}


This paper investigates the asymptotic behaviors of time-harmonic acoustic waves generated by an incident wave illuminating inhomogeneous medium inclusions with high-contrast material parameters. 
We derive sharp asymptotic estimates and obtain several effective acoustic obstacle scattering models when the material parameters take extreme values. 
The results clarify the connection between inhomogeneous medium scattering and obstacle scattering for acoustic waves, providing a clear criterion for identifying the boundary conditions of acoustic obstacles in practice.
The contributions of this paper are twofold.
First, we provide a rigorous mathematical characterization of the classical sound-hard and sound-soft obstacle scattering models. We demonstrate that a sound-hard obstacle can be viewed as an inhomogeneous medium inclusion with infinite mass density, while a sound-soft obstacle corresponds to an inclusion with zero mass density and zero bulk modulus. Second, we introduce two novel acoustic obstacle scattering models when the mass density of the inclusion degenerates to zero. These new models offer a fresh perspective on considering inhomogeneous medium inclusions with high-contrast material parameters.



%

\vspace{0.3cm}

\medskip


\end{abstract}
 
\maketitle

\section{background introduction and mathematical motivation}

The identification of boundary conditions on the acoustic obstacle surface or the interface between homogeneous background medium and inhomogeneous penetrable inclusion plays a crucial role in the mathematical modeling of acoustic wave scattering. When studying the acoustic interactions with an inhomogeneous medium inclusion with high-contrast material parameters, it is convenient and reasonable to approximate the inclusion as an impenetrable obstacle with an appropriate boundary condition. This approximation simplifies the analysis of the physical system, making the identification of boundary conditions on the obstacle particularly important. In this paper, we quantitatively characterize the asymptotic behaviors of acoustic waves interacting with a high-contrast inhomogeneous medium inclusion and prove that the inhomogeneous medium scattering converges to several obstacle scattering scenarios as the material parameters degenerate to different extreme values.

The primary purpose of this paper is to clarify the connection between inhomogeneous medium scattering and obstacle scattering for acoustic waves. This study is motivated by the reconstruction of complex scatterers with large mass density parameters considered in \cite{Liu2011}. 
For inhomogeneous medium scattering, the transmission boundary conditions between  background medium and inhomogeneous  inclusion describe how acoustic waves behave when they cross the interface \cite{pierce2019acoustics}. It is physically assumed that the acoustic pressure and particle velocity must be continuous on both sides of the boundary, meaning that the acoustic pressure and particle velocity on one side of the boundary must equal those on the other side. 
These continuities result from the conservation of mass and momentum. If this were not the case, it would produce an infinite force or acceleration acting at the boundary, leading to unrealistic physical situations.
When the inhomogeneous inclusion takes extreme material parameters, the acoustic wave generally cannot pass through the interface, causing the inhomogeneous medium acoustic scattering to degenerate into acoustic obstacle scattering. The corresponding transmission boundary condition on the interface also becomes sound-hard, sound-soft, or other boundary conditions depending on the types of extreme material parameters. 
In this paper, we derive sharp estimates of the acoustic wave when the mass density and bulk modulus degenerate to zero or infinity, respectively. 
When the mass density degenerates to infinity, the inhomogeneous medium inclusion becomes a sound-hard obstacle with a homogeneous Neumann boundary condition. Conversely, a sound-soft obstacle can be seen as an inhomogeneous medium inclusion with zero mass density and zero bulk modulus. 
Additionally, we introduce two novel obstacle scattering models when the mass density degenerates to zero, which are entirely new in acoustic wave scattering.
These results not only provide a mathematical criterion for identifying boundary conditions in real physical situations but also offer a clear guidance in the construction of acoustic materials for specific applications.

Another important motivation in this paper originates from the gradient estimate problem among multiple nearly touching inclusions in the theory of composite materials \cite{budiansky1984high, Markenscoff1992}.
 In the context of composite materials, the blowup order of the gradient of the underlying field represents the degree of field or stress concentration. Most linear fracture models assume that fracturing will occur at points with extreme field or stress concentration. Capturing the singular behaviors of the gradient is significant from both theoretical and engineering perspectives.
When the material parameters are bounded away from zero and infinity, the gradient is generally uniformly bounded, regardless of how close the inclusions approach each other \cite{Bonnetier2000, Li2003}. 
However, gradient blowup typically occurs in the narrow interval between the inclusions when the material parameters degenerate to zero or infinity. Numerous studies focus on electrostatics \cite{li2023gradient, kang2014characterization}, linear elasticity \cite{kang2019quantitative, chen2021estimates}, and more general elliptic equations \cite{Dong2021, Dong2024}.
There are also studies concerning field concentration for wave fields in the frequency regime where the maximum principle fails. This research is motivated by the theory of composite materials in photonics and phononics, where high-contrast building blocks are used to manipulate waves \cite{ammari2018minnaert, li2022minnaert}. The related studies usually concern the quasi-static case, namely $\omega \cdot \mathrm{diag}(\Omega) \ll 1$ \cite{Deng2022, Deng2022a}, and the specific subwavelength resonant regime \cite{ammari2020close, li2023interaction}.
Since material and geometric irregularities are both crucial for the occurrence of gradient blowup, it is highly meaningful to consider inhomogeneous medium acoustic scattering with extreme material parameters. We would investigate the related gradient estimate problem for multiple high-contrast and nearly-touching inclusions in future research.


The rest of this paper is organized as follows. 
In Section 2, we present the mathematical model of acoustic wave scattering and the main results in this paper. 
In Section 3, we provide the complete proof of the main results, based on some a priori estimates and well-posedness results. 
In Section 4, we consider the acoustic wave scattering problem for radial geometry, which reinforces our findings. 
For completeness, we collect some useful formulas for layer potentials and asymptotic formulas of spherical Bessel and Hankel functions in Appendix A. 
\section{Mathematical setups and main results} \label{1dot2}
In this section, we present the mathematical formulation and main results in this paper.
Let us consider the time-harmonic acoustic wave scattering problem arising from an incident field $u^i$ and a  $C^{2,\alpha}$ embedded inhomogeneous medium inclusion $(\Omega;\tilde{\rho},\tilde{\kappa})$ in a uniformly homogeneous medium $\bbr^N, N=2,3$. 
Here we denote by  $(\tilde{\rho},\tilde{\kappa})$ the related mass density and bulk modulus inside $\Omega$, respectively,  and $(\rho,\kappa)$ are  for the homogeneous background medium $\bbr^N \backslash \overline{\Omega}$.
we introduce several dimensionless parameters describing the contrast of material parameters in different regions, namely
\begin{equation*}
\delta = \frac{\tilde{\rho}}{\rho} \coma \tau = \frac{\tilde{\kappa}}{\kappa}  \coma n = \sqrt{\frac{\delta}{\tau}}.
\end{equation*} 
We also include the possibility that the inclusion is absorbing, i.e.
$$\tilde{\kappa} = \kappa_0 (\alpha - \beta\mi) \coma  \beta>0 \coma \mathrm{i}:=\sqrt{-1},$$
where $\kappa_0$ is the magnitude of bulk modulus and $\beta$ is a positive damping coefficient describing the thermal losses.
The incident field $u^i \in H_{loc}^1(\bbr^N)$ is an entire solution to homogeneous Helmholtz equation,
$$\Delta u^i(\mb{x}) + k^2 u^i(\mb{x}) =  0  \coma \mb{x} \in \bbr^N ,$$ 
where $\omega$ is the angular frequency and $k := \omega \sqrt{\rho / \kappa}$ describes the wavenumber of incident field. The incident field $u^i$ impinging on the inclusion $\Omega$ gives rise to the scattered field $u^s_{_\delta}$. Let $u_{_\delta}:=u^i +u^s_{_\delta}$ denotes the total field. 
The forward acoustic scattering problem is described by the following Helmholtz system,
\begin{equation}\label{9-1}
\begin{cases}
 \displaystyle \nabla \cdot \left(\frac{1}{\rho} \nabla u_{_\delta} \right)  + \omega^2 \frac{1}{\kappa} u_{_\delta} = 0  \hspace{45pt} \text{in} \quad \bbr^N \backslash \overline{\Omega}, \medskip \\
\displaystyle \nabla \cdot \left(\frac{1}{\tilde{\rho}} \nabla u_{_\delta} \right)  + \omega^2 \frac{1}{\tilde{\kappa}} u_{_\delta} = 0   \hspace{45pt} \mathrm{in} \quad \Omega , \medskip \\
u_{_{\delta}}|_+ = u_{_{\delta}}|_-   \coma \left. \frac{1}{\rho}\frac{\partial u_{_{\delta}}}{\partial v}\right|_+ = \left.\frac{1}{\tilde{\rho}}\frac{\partial u_{_{\delta}}}{\partial v}\right|_- \hspace{18pt} \textrm{on} \quad \partial \Omega, \medskip \\
u^s_{_\delta} \quad  \textrm{satisfies the Sommerfeld radiation condition},
\end{cases}
\end{equation}
where $\partial_{\nu}u = \hat{\mb{x}} \cdot \nabla u$ with $\hat{\mb{x}} = \mb{x}/|\mb{x}| \in \mathbb{S}^{N-1}$ and $|_{\pm}$ denotes the limits from outside and inside $\Omega$.
The Sommerfeld radiation condition implies that
\begin{equation*}
\lim_{r \rightarrow \infty} |\mb{x}|^{\frac{N-1}{2}}\left(\frac{\partial u^s_{_{\delta}} }{\partial \nu}-\mi k u^s_{_{\delta}}  \right) = 0,
\end{equation*}
where $\mi $ denotes the imaginary unit and the limit characterizes outgoing nature of the scattered field in infinity.  We refer to \cite{kress2013linear} for the well-posedness of the scattering problem \eqref{9-1} in $H_{loc}^1(\bbr^N)$.
Furthermore, the total field has the following asymptotic expansions:
\begin{equation*} 
u_{_\delta}(\mb{x}) = u^i(\mb{x}) + \frac{e^{\mi k |\mb{x}|}}{|\mb{x}|^{(N-1)/2}}u_{_\delta,\infty}(\mb{x}) + \mathcal{O}\left(\frac{1}{|\mb{x}|^{(N+1)/2}} \right) \quad \text{\rm as} \quad |\mb{x}| \rightarrow \infty,
\end{equation*}
and it holds uniformly for all directions $\hat{\mb{x}}$. The complex-value function $u_{_\delta,\infty}$, defined on the unit sphere $\mathbb{S}^{N-1}$ , is known as the far field pattern and encodes the scattering information caused by the perturbation of the incident field $u^i$ due to the scatterer $(\Omega;\tilde{\rho}, \tilde{\kappa})$.
Though  the Helmholtz equation  models the transverse electromagnetic field in two dimensions, it is insightful to study  two-dimensional case  and after adapting the physical parameters one could use this setting to study electromagnetic wave.

In this paper, we are concerned with the asymptotic behaviors of the acoustic wave  when the material parameters are high-contrast  between inside and outside $\Omega$.  By choosing appropriate physical units, we may assume that the parameters $(\rho, \kappa)$ outside the inclusion, $(\alpha,\beta)$  and the angular frequency $\omega$ are of order one.
\begin{theorem}\label{mainresult}
Consider the inhomogeneous medium scattering problem \eqref{9-1} and let $u_{_\delta}  \in H^1_{loc}(\bbr^N)$ be the unique solution, then
\begin{enumerate}
\item when $\delta \rightarrow 0$, $u_{_\delta}$  converges to $v$ in $ H^1_{loc}(\bbr^N \backslash \overline{\Omega})$ where $v$ is the unique solution to the following obstacle scattering:
\begin{equation}\label{9-3}
\begin{cases}
\displaystyle \Delta v + k^2 v = 0   \hspace{80.5pt} \mathrm{in}  \quad  \bbr^N \backslash \overline{\Omega} , \\
\displaystyle \int_{\partial \Omega}  \frac{1}{\rho}\frac{\partial v}{\partial \nu}  \dx{s} +\frac{\omega^2}{\tilde{\kappa}} \mathrm{V}(\Omega) v = 0 \quad \mathrm{on} \quad  \partial \Omega,\\
\displaystyle v-u^i   \quad \text{\rm satisfies the Sommerfeld radiation condition}.
\end{cases}
\end{equation}
where $\mathrm{V}(\Omega)$ is the volume of the inhomogeneous medium inclusion $\Omega$.
\item when $\delta \rightarrow 0$ and $|\tau| \rightarrow \infty$, $u_{_\delta}$ converges to $t$ in $ H^1_{loc}(\bbr^N \backslash \overline{\Omega})$ where $t$ is the unique solution to the following  obstacle scattering:
\begin{equation}\label{9-222}
\begin{cases}
\displaystyle \Delta t + k^2 t = 0   \quad \mathrm{in}  \quad  \bbr^N \backslash \overline{\Omega} , \\
\displaystyle  t=c \hspace{49pt} \mathrm{on} \quad \partial \Omega, \\
\displaystyle t-u^i   \quad \text{\rm satisfies the Sommerfeld radiation condition},
\end{cases}
\end{equation}
where $c$ is some unknown constant determined by 
$$
\displaystyle \int_{\partial \Omega}  \frac{\partial t}{\partial \nu} \dx{s} = 0.$$
\item when $\delta \rightarrow 0$ and $|\tau| \rightarrow 0$, $u_{_\delta}$ converges to $w$ in $ H^1_{loc}(\bbr^N \backslash \overline{\Omega})$ where $w$ is the unique solution to the following sound-soft obstacle scattering:
\begin{equation}\label{9-22}
\begin{cases}
\displaystyle \Delta w + k^2 w = 0   \quad \mathrm{in}  \quad  \bbr^N \backslash \overline{\Omega} , \\
\displaystyle  w =0 \hspace{52pt} \mathrm{on} \quad \partial \Omega, \\
\displaystyle w-u^i   \quad \text{\rm satisfies the Sommerfeld radiation condition}.
\end{cases}
\end{equation}
\item when $\delta \rightarrow \infty$, $u_{_\delta}$ converges to $u$ in $ H^1_{loc}(\bbr^N \backslash \overline{\Omega})$ where $u$ is the unique solution to the following sound-hard  obstacle scattering:
\begin{equation}\label{9-2}
\begin{cases}
\displaystyle \Delta u + k^2 u = 0   \quad \mathrm{in}  \quad  \bbr^N \backslash \overline{\Omega} ,\\
\displaystyle \frac{\partial u}{\partial \nu} =0 \hspace{42pt} \mathrm{on}\quad \partial \Omega,\\
\displaystyle u-u^i   \quad \text{\rm satisfies the Sommerfeld radiation condition}.
\end{cases}
\end{equation}
\end{enumerate}
\end{theorem}
There are some remarks in order about the main theorem.
\begin{remark}
These convergences have sharp estimates:
\begin{align*}
\|u_{_\delta} - v\| &\leqslant  C  \delta^{\frac{1}{2}} \|u^i\| \coma \hspace{40pt}
\|u_{_\delta} - t\| \leqslant  C  (\delta^{\frac{1}{2}} + |\tau|^{-\frac{1}{2}}) \|u^i\| ,  \\
\|u_{_\delta} - w\|&\leqslant  C  (\delta^{\frac{1}{2}} + |\tau|^{\frac{1}{2}}) \|u^i\| \coma
\|u_{_\delta} - u\| \leqslant  C  \delta^{-\frac{1}{2}} \|u^i\| ,
\end{align*}
where $C$ is some constant and  $\|\cdot\|:= \|\cdot\|_{H^1_{loc}(\bbr^N\backslash \overline{\Omega})}$.
\end{remark}

\begin{remark}
Generally speaking, the types of boundary conditions primarily depend on the contrast of mass density. When the mass density approaches infinity, the inhomogeneous medium inclusion typically becomes a sound-hard obstacle with a homogeneous Neumann boundary condition, regardless of the bulk modulus value. Conversely, when the mass density approaches zero, the boundary value on the inhomogeneous medium inclusion is generally constant. The identification of this constant Dirichlet value depends on the contrast of the bulk modulus.
\end{remark}

\begin{remark}

In addition to introducing two novel boundary conditions for acoustic obstacle scattering, we have also recovered the classical sound-hard and sound-soft obstacle scattering by considering extreme values of the material parameters. The classical impedance obstacle scattering is, in fact, a generalized model of \eqref{9-2} and can be derived using the result of \eqref{9-2} together with the factorization method. To obtain this boundary condition, it is necessary to modify the inhomogeneous medium scattering described in \eqref{9-1} by introducing an extremely thin, highly conductive layer between the inhomogeneous medium inclusion and the background medium (see \cite{bondarenko2016factorization}). However, we do not choose  to include this case, as it falls outside the scope of the inhomogeneous medium scattering model in \eqref{9-1}.
\end{remark}

\begin{remark}
An important and intriguing open problem is to study the gradient blowup between multiple nearly touching inclusions in the Helmholtz system, which is motivated by the stability of acoustic materials, particularly in high-contrast construction materials. The gradient of the total field is generally expected to blow up in the narrow regions between obstacles due to material and geometric irregularities. There are few studies on the Helmholtz system because it presents significant challenges compared to the Laplace system. We plan to explore this problem in future research.
\end{remark}

\section{The main proofs}
In this section, we provide the complete proof of Theorem \ref{mainresult}. Before doing so, we first present some a priori estimates for the inhomogeneous medium scattering problem \eqref{9-1}. These estimates describe the behaviors of the total field inside the inclusion when the mass density and bulk modulus take extreme values. We also address the well-posedness of these obstacle scattering models, particularly those with entirely new boundary conditions.

\subsection{A priori estimates to the medium scattering}

We note that the Sobolev  space $H_{loc}^1(\bbr^N \backslash \overline{\Omega})$ is the space of all functions $f:\bbr^N \backslash \overline{\Omega} \rightarrow \mathbb{C}$ such that $f \in H^1(B_R \backslash \overline{\Omega})$ for all open balls $B_R$ with radius $R$ containing $\overline{\Omega}$.

\begin{lemma}
Let $u_{_{\delta}} \in H_{loc}^1(\bbr^N)$ be the unique solution to the inhomogeneous medium scattering \eqref{9-1} and  $d_0 \in (0,1)$  be a positive number, then the following estimates hold. 
If the parameter  $\delta> d_0^{-1}$, it holds that
\begin{equation}\label{aprior1}
\delta^{-\frac{1}{2}}\|u_{_{\delta}}\|_{H^1(\Omega)} \leqslant C \|u_{_{\delta}}\|_{H^1(B_R\backslash \overline{\Omega})},
\end{equation}
and if $0<\delta< d_0$, it holds that
\begin{equation}\label{aprior2}
\|\nabla u_{_{\delta}}\|_{L^2(\Omega)} \leqslant C \delta^{\frac{1}{2}} \|u_{_{\delta}}\|_{H^1(B_R\backslash \overline{\Omega})} ,
\end{equation}
where $C$ is some positive constant independent of $\delta$. Furthermore, if $0<\delta< d_0$ and $0< |\tau| < d_0$, it holds that
\begin{equation}\label{aprior3}
\|u_{_\delta}\|_{H^1(\Omega)} \leqslant C(\delta + |\tau|)^{\frac{1}{2}} \|u_{_{\delta}}\|_{H^1(B_R\backslash \overline{\Omega})} ,
\end{equation}
where $C$ is some positive constant independent of $\delta$ and $\tau$. Finally for all the possible cases above, the total field $u_{_{\delta}} $ can be controlled by the incident field $u^i$, i.e.,
\begin{align}\label{9-4}
\|u_{_{\delta}}\|_{H^1(B_R\backslash\overline{\Omega})} \leqslant C \|u^i\|_{H^1(B_R\backslash\overline{\Omega})}.
\end{align}
\end{lemma}
\begin{proof}
For the Helmholtz system \eqref{9-1}, multiplying by $\overline{u_{_{\delta}}}$ and integrating inside $\Omega$  and outside $B_R \backslash \Omega$ for a sufficient large $R$ containing $\overline{\Omega}$  yield
\begin{align*}
&\int_{\Omega} \nabla \cdot \frac{1}{\tilde{\rho}} \nabla u_{_{\delta}}\overline{u_{_{\delta}}} + \frac{\omega^2}{\tilde{\kappa}} u_{_{\delta}}\overline{u_{_{\delta}}} \dx{x} + \int_{B_R\backslash\overline{\Omega}} \nabla \cdot \frac{1}{\rho} \nabla u_{_{\delta}}\overline{u_{_{\delta}}} + \frac{\omega^2}{\kappa} u_{_{\delta}}\overline{u_{_{\delta}}} \dx{x} \\
=& \int_{\Omega}  -\frac{1}{\tilde{\rho}}|\nabla u_{_{\delta}}|^2+ \frac{\omega^2}{\tilde{\kappa}}|u|^2 \dx{x} + \int_{\partial \Omega}\frac{1}{\tilde{\rho}}\left.\frac{\partial u_{_{\delta}}}{\partial \nu}\right|_-\left.\overline{u_{_{\delta}}}\right|_-\dx{s} + \int_{{B_R\backslash\overline{\Omega}}} -\frac{1}{\rho}|\nabla u_{_{\delta}}|^2+ \frac{\omega^2}{\kappa}|u_{_{\delta}}|^2 \dx{x}  \\
&+ \int_{\partial B_R}\frac{1}{\rho}\frac{\partial u_{_{\delta}}}{\partial \nu}\overline{u_{_{\delta}}}\dx{s} - \int_{\partial \Omega}\frac{1}{\rho}\left.\frac{\partial u_{_{\delta}}}{\partial \nu}\right|_+ \left.\overline{u_{_{\delta}}}\right|_+\dx{s} \\
=&-\int_{\Omega}  \frac{1}{\tilde{\rho}}|\nabla u_{_{\delta}}|^2 - \omega^2(a+\mi b)|u_{_{\delta}}|^2 \dx{x} - \int_{{B_R\backslash\overline{\Omega}}} \frac{1}{\rho}|\nabla u_{_{\delta}}|^2 - \frac{\omega^2}{\kappa}|u_{_{\delta}}|^2 \dx{x} + \int_{\partial B_R}\frac{1}{\rho}\frac{\partial u_{_{\delta}}}{\partial \nu}\overline{u_{_{\delta}}} \dx{s} \\
=& 0  ,
\end{align*}
where $a$ and $b$ are defined as
$$a = \frac{1}{|\tau|}\frac{\alpha}{\kappa\sqrt{\alpha^2+\beta^2}} \coma b =   \frac{1}{|\tau|} \frac{\beta}{\kappa\sqrt{\alpha^2+\beta^2}}. $$
Taking the imaginary part from the above equation, we can get
\begin{align*}
\omega^2 b \int_{\Omega} |u_{_{\delta}}|^2 \dx{x} + \Im\int_{\partial B_R}\frac{1}{\rho}\frac{\partial u_{_{\delta}}}{\partial \nu}\overline{u_{_{\delta}}}\dx{s} = 0.
\end{align*}
Hence, it holds from the Schwartz inequality and trace theorem that
\begin{equation}\label{udelta}
\begin{aligned}
\|u_{_{\delta}}\|_{L^2(\Omega)}^2 &= \frac{1}{\omega^2 b}\Im\int_{\partial B_R}\frac{1}{\rho}\overline{\frac{\partial u_{_{\delta}}}{\partial \nu}}u_{_{\delta}}\dx{s} \leqslant C\frac{1}{\omega^2 b\rho} \|\partial_\nu u_{_{\delta}}\|_{H^{-\frac{1}{2}}(\partial B_R)} \| u_{_{\delta}}\|_{H^{\frac{1}{2}}(\partial B_R)} \\
&\leqslant C\frac{1}{\omega^2 b\rho} \|u_{_{\delta}}\|^2_{H^1(B_R\backslash \overline{\Omega})}.
\end{aligned}
\end{equation}
Similarly, we can get by taking the real part 
\begin{align*}
&-\int_{\Omega}  \frac{1}{\tilde{\rho}}|\nabla u_{_{\delta}}|^2 - \omega^2 a \int_{\Omega}|u_{_{\delta}}|^2 \dx{x} - \int_{{B_R\backslash\overline{\Omega}}} \frac{1}{\rho}|\nabla u_{_{\delta}}|^2 - \frac{\omega^2}{\kappa}|u_{_{\delta}}|^2 \dx{x} + \Re\int_{\partial B_R}\frac{1}{\rho}\frac{\partial u_{_{\delta}}}{\partial \nu}\overline{u_{_{\delta}}}\dx{s} \\
&= 0.
\end{align*}
Direct verification shows that
\begin{align*}
\frac{1}{\delta}\|\nabla u_{_{\delta}}\|_{L^2(\Omega)}^2 &= \omega^2\rho a\int_{\Omega}|u_{_{\delta}}|^2 \dx{x} + \int_{{B_R\backslash\overline{\Omega}}} -|\nabla u_{_{\delta}}|^2+ \frac{\omega^2\rho}{\kappa}|u_{_{\delta}}|^2 \dx{x} + \Re\int_{\partial B_R}\frac{\partial u_{_{\delta}}}{\partial \nu}\overline{u_{_{\delta}}}\dx{s} \\
&=\omega^2\rho a \|u_{_{\delta}}\|^2_{L^2(\Omega)} - \|\nabla u_{_{\delta}}\|^2_{L^2(B_R \backslash \overline{\Omega})} + k^2\| u_{_{\delta}}\|^2_{L^2(B_R \backslash \overline{\Omega})} + \Re\int_{\partial B_R}\frac{\partial u_{_{\delta}}}{\partial \nu}\overline{u_{_{\delta}}}\dx{s}  \\
&\leqslant C\left(\omega^2\rho a\|u_{_{\delta}}\|_{L^2(\Omega)}^2  + \|u_{_{\delta}}\|_{H^1(B_R\backslash \overline{\Omega})}^2 \right).
\end{align*}
where $C$ is some constant only depending on $k,\alpha,\beta, \Omega$.

We are in a position to derive the corresponding asymptotic estimates when the material parameters take extreme values. Firstly if $\delta > d_0^{-1} >1$, it holds that 
\begin{align*}
\frac{1}{\delta}\|u_{_{\delta}}\|_{H^1(\Omega)}^2 &\leqslant C\left(\omega^2\rho a\|u_{_{\delta}}\|_{L^2(\Omega)}^2  + \delta^{-1}\|u_{_{\delta}}\|_{L^2(\Omega)}^2 + \|u_{_{\delta}}\|_{H^1(B_R\backslash \overline{\Omega})}^2 \right) \\
&\leqslant C\|u_{_{\delta}}\|_{H^1(B_R\backslash \overline{\Omega})}^2.
\end{align*}
We obtain the second estimate \eqref{aprior1}.
If $0< \delta < d_0<1$ , then
\begin{equation}\label{boundary}
\|\nabla u_{_{\delta}}\|_{L^2(\Omega)}^2 \leqslant C \delta\left(\frac{\alpha}{\beta}\|u_{_{\delta}}\|_{L^2(\Omega)}^2  + \|u_{_{\delta}}\|_{H^1(B_R\backslash \overline{\Omega})}^2 \right) \leqslant C\delta \|u_{_{\delta}}\|_{H^1(B_R\backslash \overline{\Omega})}^2.
\end{equation}
We obtain the third estimate \eqref{aprior2}. If we further assume  $0< |\tau| < d_0<1$, then
\begin{equation*}
\|u_{_{\delta}}\|_{H^1(\Omega)}^2 = \|u_{_{\delta}}\|_{L^2(\Omega)}^2 + \|\nabla u_{_{\delta}}\|_{L^2(\Omega)}^2 \leqslant C (\delta + |\tau|) \|u_{_{\delta}}\|_{H^1(B_R\backslash \overline{\Omega})}^2.
\end{equation*}
We obtain the last estimate \eqref{aprior3}.

The validity of  estimate \eqref{9-4} is in fact a result of energy conservation. It states that the energy of the total field can be controlled by that of the incident field. We follow the idea presented in \cite{Liu2011} to prove it. Here we choose to consider the case $0< \delta <d_0 <1$ and other cases can be proved in a similar manner.   

Let $\{u^{i,n}\}_{n \in \bbn}$ be a sequence of normalized incident fields satisfying $\|u^{i,n}\|_{H^1(B_R\backslash \overline{\Omega})} =1$. 
Without loss of generality, if \eqref{9-4} does not hold, we assume that the produced total fields $\{ u^n_{_\delta}\}_{n \in \bbn}$ satisfy
$$\|u^n_{_\delta}\|_{H^1(B_R\backslash \overline{\Omega})} \geqslant n \longrightarrow \infty \quad \text{\rm as}  \quad n \rightarrow \infty.$$ 
We claim that this is a contradiction. Let
\begin{equation*}
v^{i,n} = \frac{u^{i,n}}{\|u^n_{_\delta}\|_{H^1(B_R\backslash \overline{\Omega})}} \coma v^{n} = \frac{u^{n}_{\delta}}{\|u^n_{_\delta}\|_{H^1(B_R\backslash  \overline{\Omega})}} \coma  n \in \bbn.
\end{equation*}
Since the system \eqref{9-1} is linear, it holds that $v^{n}$ is the total field to \eqref{9-1} illuminated by the incident field $v^{i,n}$. It is readily observed that
\begin{equation}\label{normal}
\|v^{i,n}\|_{H^1(B_R\backslash \overline{\Omega})} \leqslant \frac{1}{n} \rightarrow 0  \coma \|v^n\|_{H^1(B_R\backslash \overline{\Omega})} \equiv 1 \quad \text{\rm as}  \quad n \rightarrow \infty.
\end{equation}
The completely same argument as \eqref{boundary} yields
\begin{equation*}
\|\nabla v^n\|_{L^2(\Omega)}  \leqslant C\delta^{\frac{1}{2}} \|v^n\|_{H^1(B_R\backslash \overline{\Omega})} = C \delta^{\frac{1}{2}}.
\end{equation*}
By using the transmission boundary condition and Green's formula, we can get
\begin{align*}
&\int_{\partial \Omega} \left.\frac{1}{\rho}\frac{\partial v^n}{\partial \nu}\right|_+ \dx{s}  + \frac{\omega^2}{\tilde{\kappa}} \mathrm{V}(\Omega)v^n|_+ = \int_{\partial \Omega} \left.\frac{1}{\tilde{\rho}}\frac{\partial v^n}{\partial \nu}\right|_- \dx{s}  + \frac{\omega^2}{\tilde{\kappa}} \mathrm{V}(\Omega)v^n|_- \\
=& \frac{\omega^2}{\tilde{\kappa}}\mathrm{V}(\Omega) \left(v^n|_- - \int_\Omega v^n \dx{x}\right) =  \frac{\omega^2}{\tilde{\kappa}}\mathrm{V}(\Omega) (v^n|_- - \tilde{v}_{_\Omega}^n),
\end{align*}
where $\tilde{v}^n_{_\Omega}$ is the integral average of $v^n$ over $\Omega$.
Then taking the Poincar\'e-Wirtinger inequality and trace theorem we can get
\begin{equation*}
\left\|\int_{\partial \Omega} \left.\frac{1}{\rho}\frac{\partial v^n}{\partial \nu}\right|_+ \dx{s}  + \frac{\omega^2}{\tilde{\kappa}} \mathrm{V}(\Omega)v^n|_+ \right\|_{H^{\frac{1}{2}}(\partial \Omega)} \leqslant C\|\nabla v^n\|_{L^2(\Omega)} \leqslant C\delta^{\frac{1}{2}}.
\end{equation*}
Therefore from Lemma \ref{9-7}, there exists a constant $0<d_0 <1$ and sufficiently large $N$ such that for all $n>N$ and $\delta < d_0$,
\begin{equation*}
\|v^n\|_{H^1(B_R\backslash \overline{\Omega})} \leqslant   C\delta^{\frac{1}{2}} + \|v^{i,n}\|_{H^1(B_R\backslash \overline{\Omega})} < 1,
\end{equation*}
which  contradicts \eqref{normal} and hence \eqref{9-4} holds.
\end{proof}

\subsection{The wellposedness to the obstacle scattering}
In this subsection, we prove the well-posedness of the obstacle scattering problems presented in Theorem \ref{mainresult}. The primary mathematical tools employed are layer potential techniques and Riesz-Fredholm theory. The well-posedness results could also be obtained using variational methods for elliptic boundary value problems, which are particularly straightforward for the scattering problems in \eqref{9-222} and \eqref{9-22}. However, to avoid introducing additional mathematical frameworks, we choose to base our analysis entirely on potential theory rather than invoking more general results. The necessary mathematical definitions and operator properties of layer potentials are provided in Appendix \ref{layer}.

The sound-soft \eqref{9-22} and sound-hard \eqref{9-2} obstacle scattering problems are classical and their well-posedness results can be found in many mathematical texts, such as \cite{kress2013linear}. For completeness, we also include the well-posedness results here. In this subsection, our primary focus is on the well-posedness of the two new obstacle scattering models, \eqref{9-3} and \eqref{9-222}. Let $v^s, t^s, w^s$ and $u^s$ denote the scattered fields for the exterior obstacle scattering problems \eqref{9-3}, \eqref{9-222}, \eqref{9-22}, and \eqref{9-2}, respectively. After renaming the unknown functions to represent the scattered fields, we consider a more general case of these direct scattering problems.

\begin{lemma}\label{9-7}
Given a function  $f \in H^{\frac{1}{2}}(\partial \Omega)$ and we consider the following exterior boundary value problem:
\begin{equation}\label{d1}
\begin{cases}
\displaystyle \Delta v^s +k^2v^s = 0 \hspace{68pt} \text{\rm in} \quad \bbr^N \backslash \overline{\Omega},  \\
\displaystyle \int_{\partial \Omega} \frac{1}{\rho}\frac{\partial v^s}{\partial \nu}  + \frac{\omega^2}{\tilde{\kappa}} \mathrm{V}(\Omega)v^s = f \quad \text{\rm on} \quad  \partial \Omega, \\
\displaystyle  v^s \quad \text{\rm satisfies the Sommerfeld radiation condition}.
\end{cases}
\end{equation}
 Then there exists a unique solution $v^s \in {H^1_{loc}(\bbr^N \backslash \overline{\Omega})}$ and the solution satisfies
\begin{equation*}
\|v^s\|_{H^1_{loc}(\bbr^N \backslash \overline{\Omega})}  \leqslant C \|f\|_{H^{\frac{1}{2}}(\partial \Omega)},
\end{equation*}
 where  $C$ is some constant only depending on $k , \tilde{\kappa}$ and $\Omega$.
\end{lemma}
\begin{proof}
By the layer-potential technique, the solution to  \eqref{d1} can be represented as the summation of single- and double-layer operators:
\begin{equation}\label{solution}
v^s =  \int_{\partial \Omega} -i\iota \Phi(x,y)\psi(y) + \frac{\partial \Phi(x,y)}{\partial \nu } \psi(y) \dx{s}  \coma x \in\bbr^N \backslash \Omega,
\end{equation}
where $\iota >0$ is some positive constant and $\psi \in H^{\frac{1}{2}}(\partial \Omega)$. we define $b_0^{-1}: = \frac{\omega^2\rho}{\tilde{\kappa}}\mathrm{V}(\Omega)$ for simplicity. $w^s$  satisfies the boundary condition and by the jump relation (see Appendix \ref{layer}) it holds that
\begin{align}\label{integral}
\psi - \mi\iota S\psi  + K\psi  +b_0\langle\mi \iota(I-K^*)\psi +T\psi,1\rangle_{\partial \Omega} &= 2b_0\rho f,
\end{align}
where $\langle \cdot , \cdot\rangle_{\partial \Omega}$ is the dual product in $H^{-\frac{1}{2}}(\partial \Omega) \times H^{\frac{1}{2}}(\partial \Omega) $.

we know that $S,K ,K^*: H^{ \frac{1}{2}}(\partial \Omega) \rightarrow  H^{ \frac{1}{2}}(\partial \Omega)$ is compact  and $T$ is bounded from $H^{\frac{1}{2}}(\partial \Omega)$ to $H^{-\frac{1}{2}}(\partial \Omega)$.  We need some analysis for the dual product $\langle \cdot , \cdot\rangle_{\partial \Omega}$. For $\langle T\psi , 1 \rangle$, we can view $1|_{\partial \Omega}$ be a function belonging to  $H^{\frac{1}{2}}(\partial \Omega)$ since $\partial \Omega$ is a compact set and  $\langle T\psi , 1 \rangle$ is well-defined. For any function $\phi \in  H^{-\frac{1}{2}}(\partial \Omega)$, the operator $G: H^{-\frac{1}{2}}(\partial \Omega) \rightarrow \mathbb{C}$ defined by $G\phi := \langle\phi,1 \rangle$ is  a rank-one operator and hence is a compact operator. By following the similar manner, it is also true for $I$ and $K^*$ since they are bounded from $H^{ \frac{1}{2}}(\partial \Omega)$ to $H^{ \frac{1}{2}}(\partial \Omega)$. We define $F\psi := b_0\langle\mi \iota(I-K^*)\psi +T\psi,1\rangle_{\partial \Omega}$ and $F$ is a compact operator from $ H^{\frac{1}{2}}(\partial \Omega)$ to  $ H^{\frac{1}{2}}(\partial \Omega)$. The existence of solution to \eqref{d1} can be established by the Riesz-Fredholm theory for equations with a compact operator.

We first prove the uniqueness by Rellich lemma. when $f =0$, then
\begin{equation*}
\Im \int_{\partial \Omega} v^s \frac{\partial \overline{v^s}}{\partial \nu} = \Im \int_{\partial \Omega}-b_0 \int_{\partial \Omega}\frac{\partial v^s}{\partial \nu}  \frac{\partial \overline{v^s}}{\partial \nu} = \Im \left( - b_0 \left|\int_{\partial \Omega} \frac{\partial v^s}{\partial \nu} \right|^2 \right) \geqslant 0,
\end{equation*}
where we have used the fact 
$$ \Im b_0=\frac{1}{\omega^2\rho\mathrm{V}(\Omega)} \Im \tilde{\kappa} =\frac{\kappa_0}{\omega^2\rho\mathrm{V}(\Omega)}\frac{-\beta}{\alpha^2+\beta^2} \leqslant 0.$$
Hence $v^s = 0$ in $\bbr^3 \backslash \overline{\Omega}$.

When $f = 0$, the uniqueness shows that $v^s|_+ =0$.  We  from the jump relation obtain that
\begin{equation*}
v^s|_- = - \psi \coma \frac{\partial v^s}{\partial \nu} = -\mi \iota \psi
\end{equation*}
and  hence
\begin{equation*}
\mi \iota \int_{\partial \Omega} \psi^2 = \int_{\partial \Omega} \overline{v^s}|_- \frac{\partial v^s|_-}{\partial \nu} =\int_{\Omega} |\nabla v^s|^2 -k^2|v^s|^2.
\end{equation*}
Taking the imaginary part shows that $\psi = 0$. Thus we have established the uniqueness for the integral equation \eqref{integral}, i.e., injectivity of the operator $I-\mi\iota S + K + F: H^{ \frac{1}{2}}(\partial \Omega) \rightarrow  H^{ \frac{1}{2}}(\partial \Omega)$. Therefore by the Riesz-Fredholm theory, $I-\mi\iota S + K + F$ is bijective and the inverse $(I-\mi\iota S + K + F)^{-1}$ is bounded. The inhomogeneous equation \eqref{integral} possesses a solution and the solution depends on $2b_0\rho f$ in the $H^{\frac{1}{2}}(\partial\Omega)$ norm. From the representation \eqref{solution} of the solution as a combined double- and single-layer potential, with the aid of the regularity estimates in layer-potential theory, the $H^\frac{1}{2}$ dependence of the density $\psi$ on the boundary data $f$ shows that the problem  \eqref{d1} is well-posed, i.e. the $H^1_{loc}(\bbr^N \backslash \overline{\Omega})$ of the solution to \eqref{d1} can be controlled by the $H^{\frac{1}{2}}(\partial \Omega)$ of boundary data $f$. Therefore,
\begin{equation*}
\|v^s\|_{H^1_{loc}(\bbr^N\backslash \overline{\Omega})} \leqslant C\|\psi\|_{H^\frac{1}{2}(\partial \Omega)}  \leqslant C \|f\|_{H^\frac{1}{2}(\partial \Omega)}.
\end{equation*}
We have finished the proof.
\end{proof}

\begin{lemma}\label{33}
Given a function  $d \in H^{\frac{1}{2}}(\partial \Omega)$ and we consider the following exterior boundary value problem:
\begin{equation} \label{tss}
\begin{cases}
\displaystyle \Delta t^s + k^2t^s = 0 \hspace{16pt} \textit{\rm in} \quad \bbr^N \backslash \overline{\Omega}, \\
\displaystyle t^s =  d + C_0 \hspace{31pt} \textit{\rm on} \quad  \partial \Omega, \\
\displaystyle \int_{\partial \Omega} \frac{\partial t^s}{\partial \nu} \dx{s} = a_0 \quad \textit{\rm on}\quad  \partial \Omega, \\
\displaystyle t^s \quad  \text{\rm satisfies the Sommerfeld radiation condition},
\end{cases}
\end{equation}
where  $C_0$ is some unknown constant determined by the integral of normal derivative.
There exists a unique solution $t^s$ and the solution satisfies
\begin{equation*}
\|t^s\|_{H^1_{loc}(\bbr^N \backslash \overline{\Omega})} \leqslant C \left( \|d\|_{H^{\frac{1}{2}}(\partial \Omega)} + a_0 \right),
\end{equation*}
where $C$ is some constant only depending on $k,\Omega$.
\end{lemma}
\begin{proof}
We can separate the Dirichlet problem \eqref{tss} into two parts: $t^s := t_1^s +t_2^s$ where they are the unique solutions of the following systems:
\begin{equation*} 
\begin{cases}
\displaystyle \Delta t_1^s + k^2t_1^s = 0 \hspace{16pt} \textit{\rm in} \quad \bbr^N \backslash \overline{\Omega}, \\
\displaystyle t_1^s =  d  \hspace{58pt} \textit{\rm on} \quad \partial \Omega, \\
\displaystyle t_1^s \quad  \text{\rm satisfies the Sommerfeld radiation condition}.
\end{cases}
\end{equation*}
and
\begin{equation}\label{diri}
\begin{cases}
\displaystyle \Delta t_2^s + k^2t_2^s = 0 \hspace{78pt} \textit{\rm in} \quad \bbr^N \backslash \overline{\Omega}, \\
\displaystyle t^s_2 =   C_0 \hspace{113pt} \textit{\rm on} \quad  \partial \Omega, \\
\displaystyle \int_{\partial \Omega} \frac{\partial t_2^s}{\partial \nu} \dx{s} = a_0 - \int_{\partial \Omega} \Lambda t_1^s \dx{s} \quad \textit{\rm on}\quad  \partial \Omega, \\
\displaystyle t_2^s \quad  \text{\rm satisfies the Sommerfeld radiation condition}.
\end{cases}
\end{equation}
Here $\Lambda: H^{\frac{1}{2}}(\partial \Omega) \rightarrow H^{-\frac{1}{2}}(\partial \Omega)$ is the Dirichlet-to-Neumann map and $\Lambda$  is bounded \cite{Colton2019}. From the lemma \ref{dirichlet}, it holds that
\begin{equation*}
\|t_1^s\|_{H^1_{loc}(\bbr^N \backslash \overline{\Omega})} \leqslant C \|d\|_{H^{\frac{1}{2}}(\partial \Omega)}.
\end{equation*}
Since the system \eqref{diri} is a linear Helmholtz system with constant Dirichlet value $C$, we can consider the obstacle scattering system with Dirichlet value of one:
\begin{equation*}
\begin{cases}
\displaystyle \Delta t_3^s + k^2t_3^s = 0 \quad \textit{\rm in} \quad \bbr^N \backslash \overline{\Omega}, \\
\displaystyle t^s_3 =   1 \hspace{53pt} \textit{\rm on} \quad  \partial \Omega, \\
\displaystyle t_3^s \quad  \text{\rm satisfies the Sommerfeld radiation condition}.
\end{cases}
\end{equation*}
Then $t_2^s := C_0 t_3^s$ must be the unique solution to \eqref{diri}. By the wellposedness of $t_3^s$, there exists a nonzero solution $t_3^s \in H_{loc}^1 (\bbr^N \backslash \overline{\Omega})$. Then
\begin{equation*}
C_0 = \frac{a_0 - \int_{\partial \Omega} \Lambda t_1^s \dx{s}}{\int_{\partial \Omega} \Lambda t_3^s \dx{s} } := \frac{b_0}{c_0}. 
\end{equation*}
We now prove $c_0$ is nonzero by contradiction. We take $R$ large enough such that $\Omega$ is contained in $B_R$ and the Sommerfeld radiation condition yields
\begin{equation*}
\lim_{R\rightarrow \infty}\int_{\partial B_R} \left|\frac{\partial t^s_{3} }{\partial \nu}-\mi k  t^s_{3}\right|^2 \dx{s} = \lim_{R\rightarrow \infty}\int_{\partial B_R} \left|\frac{\partial t^s_{3} }{\partial \nu}\right|^2 + k^2 |t_3^s|^2 + 2k\Im \left( t^s_{3}\overline{\frac{\partial t^s_{3} }{\partial \nu}}\right) \dx{s} \rightarrow 0.
\end{equation*}
Applying Green's theorem in $B_R \backslash \overline{\Omega}$ for $t_3^s$ obtain
\begin{equation*}
\int_{\partial B_R} t^s_{3} \overline{\frac{\partial t^s_{3} }{\partial \nu}} \dx{s} - \int_{\partial \Omega} t^s_{3} \overline{\frac{\partial t^s_{3} }{\partial \nu}} \dx{s} = \int_{B_R \backslash \overline{\Omega}} k^2 |t_3^s|^2 + |\nabla t_3^s|^2 \dx{x}.
\end{equation*}
Taking the imaginary part, we can get
\begin{align*}
\lim_{R\rightarrow \infty}\int_{\partial B_R} \left|\frac{\partial t^s_{3} }{\partial \nu}\right|^2 + k^2 |t_3^s|^2 = \lim_{R\rightarrow \infty} \int_{\partial B_R}  2k\Im \left( t^s_{3}\overline{\frac{\partial t^s_{3} }{\partial \nu}}\right) \dx{s} =  \int_{\partial \Omega}  2k\Im \left( t^s_{3}\overline{\frac{\partial t^s_{3} }{\partial \nu}}\right) \dx{s}.
\end{align*}
Since $t_3^s$ is constant on the boudnary of $\Omega$ and the integral of  normal derivative is zero, it holds that
\begin{equation*}
\lim_{R\rightarrow \infty}\int_{\partial B_R}|t_3^s|^2 \dx{s} = 0.
\end{equation*} 
By the Rellich lemma, we can obtain $t_3^s = 0$ in $B_R \backslash \overline{\Omega}$, which is in contradiction with the nonzero property of assumption to  $t_3^s$. Hence it holds that
\begin{equation*}
 c_0  = \int_{\partial \Omega} \Lambda t_3^s \dx{s} \neq 0.
\end{equation*}
and 
\begin{equation*}
C_0 = \frac{a_0 - \int_{\partial \Omega} \Lambda t_1^s \dx{s}}{\int_{\partial \Omega} \Lambda t_3^s \dx{s} } \leqslant C \left(a_0 + \|d\|_{H^{\frac{1}{2}}(\partial \Omega)} \right),
\end{equation*}
where $C$ is some constant only depending on $k$ and $\Omega$.
\end{proof}

\begin{lemma}\label{dirichlet}
Consider the following Dirichlet problem:
\begin{equation*}
\begin{cases}
\displaystyle \Delta w^s + k^2w^s = 0 \hspace{26pt} \text{\rm in} \quad \bbr^N \backslash \overline{\Omega}, \\
\displaystyle w^s = h \in H^{\frac{1}{2}}(\partial \Omega) \quad \text{\rm on} \quad  \partial \Omega, \\
\displaystyle w^s \quad  \text{\rm satisfies the Sommerfeld radiation condition}.
\end{cases}
\end{equation*}
There exists a unique solution $w^s$ and the solution satisfies
\begin{equation*}
\|w^s\|_{H^1_{loc}(\bbr^N \backslash \overline{\Omega})} \leqslant C \|h\|_{H^{\frac{1}{2}}(\partial \Omega)}.
\end{equation*}
where $C$ is some constant depends on $k,\Omega$.
\end{lemma}
\begin{proof}
The uniqueness is given by the Rellich lemma since $w^s$ is identically zero on $\partial \Omega$.
We seek the solutions defined as
\begin{equation*}
w^s =  \int_{\partial \Omega} \left(\frac{\partial \Phi (x,y)}{\partial \nu_y} - \mi \zeta \Phi(x,y) \right)\psi (y) \dx{s_y} \coma x \in \bbr^N \backslash \overline{\Omega},
\end{equation*}
where $\psi \in H^{\frac{1}{2}}(\partial \Omega)$ and $\zeta$ is a real nonzero coupling parameter. 
We note that $w^s$ shares the same mathematical representation as $v^s$ and a similar argument as Lemma \ref{9-7} yields the wellposedness results.
\end{proof}
\begin{lemma}\label{9-6}
Given a function  $g \in H^{-\frac{1}{2}}(\partial \Omega)$ and we consider the following exterior boundary value problem:
\begin{equation*}
\begin{cases}
\displaystyle \Delta u^s + k^2u^s = 0 \quad \text{\rm in} \quad \bbr^N \backslash \overline{\Omega}, \\
\displaystyle \frac{\partial u^s}{\partial \nu} = g \hspace{46pt} \text{\rm on} \quad  \partial \Omega, \\
\displaystyle u^s \quad  \text{\rm satisfies the Sommerfeld radiation condition}.
\end{cases}
\end{equation*}
There exists a unique solution $u^s \in H^1_{loc}(\bbr^N \backslash \overline{\Omega}) $ and the solution satisfies
\begin{equation*}
\|u^s\|_{H^1_{loc}(\bbr^N \backslash \overline{\Omega})} \leqslant C \|g\|_{H^{-\frac{1}{2}}(\partial \Omega)},
\end{equation*}
where $C$ is some constant only depending on $k$ and $\Omega$.
\end{lemma}

\begin{proof}

The uniqueness is guaranteed by the Rellich lemma. We seek the solution $u^s$  in the form
\begin{equation}\label{integral2}
u^s = \int_{\partial \Omega} \left(\Phi(x,y)\psi(y) + \mi\eta \frac{\partial \Phi (x,y)}{\partial \nu (y)}(S_0^2\psi)(y) \right) \dx{s(y)} \coma x \in \bbr^N \backslash \overline{\Omega},
\end{equation}
where $\psi \in H^{-\frac{1}{2}}(\partial\Omega)$ and $\eta$ is a real coupling parameter $\eta \neq 0$. By $S_0$ we denote the single-layer operator in the potential theoretic limit $k =0$. The boundary condition and jump formulas for layer-potential operator yield
\begin{equation*}
(I - (K^* +\mi\eta TS_0^2))\psi = -2 g,
\end{equation*}
From the fact that $S_0^2$ is bounded  from $H^{-\frac{1}{2}}(\partial \Omega)$ to $H^{\frac{3}{2}}(\partial \Omega)$ and $T$ is bounded $H^{\frac{3}{2}}(\partial \Omega)$ to $H^{\frac{1}{2}}(\partial \Omega)$, $TS_0^2$ is compact from $H^{-\frac{1}{2}}(\partial \Omega)$ to $H^{-\frac{1}{2}}(\partial \Omega)$ with the aid of the compactness of the imbedding operators in  Soblev spaces.

Let $\psi$ be the solution to the homogeneous form of \eqref{integral2}. Then by the jump relations we can get
\begin{equation*}
-u^s|_- = \mathrm{i}\eta S_0^2\psi \coma - \frac{\partial u^s|_-}{\partial \nu} = -\psi.
\end{equation*}
By interchanging the order of integration and using Green's integral theorem, we obtain
\begin{equation*}
\mathrm{i}\eta \int_{\partial \Omega} |S_0\psi|^2\dx{x} = \mathrm{i}\eta \int_{\partial \Omega} \psi  S_0^2\overline{\psi}\dx{x} = \int_{\partial \Omega}\overline{u^s}|_- \frac{\partial u_-}{\partial \nu} \dx{x} = \int_{\Omega} |\nabla u^s|^2 -k^2|u^s|^2\dx{x}.
\end{equation*}
Hence $S_0\psi =0$ on $\partial \Omega$ follows. The single-layer potential with density $\psi$ and $k=0$ is continuous throughout the whole space and harmonic in $\Omega$ and $\bbr^N \backslash \overline{\Omega}$. By he maximum principle and jump relation  we can get $\psi =0$. Therefore, the Riesz-Fredholm theory together with the regularity estimates in layer-potential theory yield the desired results.
\end{proof}

\subsection{proof of Theorem \ref{mainresult}} 
With all the preliminary work, we can prove  the main theorem \ref{mainresult}.
\begin{proof}
When $\delta$ is sufficiently large,  we can readily get 
\begin{equation*}
\left\|\left.\frac{\partial u_{_{\delta}}}{\partial \nu}\right|_+\right\|_{H^{-\frac{1}{2}}(\partial \Omega)} = \left\|\left.\frac{\rho}{\tilde{\rho}}\frac{\partial u_{_{\delta}}}{\partial \nu}\right|_-\right\|_{H^{-\frac{1}{2}}(\partial \Omega)} \leqslant C\delta^{-1}\|u_{_{\delta}}\|_{H^1(\Omega)}\leqslant C \delta^{-\frac{1}{2}}\|u^i\|_{H^1_{loc}(\bbr^N\backslash \overline{\Omega})},
\end{equation*}
where we have used the inequality \eqref{aprior1}. Then let $u_{inf} = u_{_{\delta}} - u$ where $u$ is the unique solution to sound-hard obstacle scattering \eqref{9-2}, we can see that $u_{inf}$ satisfies Lemma \ref{9-6} with the boundary condition 
$$g :=\frac{\partial u_{_\delta}}{\partial \nu} ,$$
and it holds that
\begin{equation*}
\|u_{inf}\|_{H^1_{loc}(\bbr^N\backslash \overline{\Omega})} \leqslant C\left\|\left.\frac{\partial u_{_{\delta}}}{\partial \nu}\right|_+\right\|_{H^{-\frac{1}{2}}(\partial \Omega)} \leqslant C  \delta^{-\frac{1}{2}} \|u^i\|_{H_{loc}^{1}(\bbr^N\backslash \overline{\Omega})}.
\end{equation*}
We have proved $u_{_\delta} \rightarrow u $ in $H^1_{loc}(\bbr^N)$ as $\delta \rightarrow \infty$.  

When $\delta$ is sufficiently small, by using the transmission boundary condition in \eqref{9-1} and Green's formula we can get
\begin{equation*}
\begin{aligned}
\int_{\partial \Omega} \left.\frac{1}{\rho}\frac{\partial u_{_{\delta}}}{\partial \nu}\right|_+\dx{s}  + \frac{\omega^2}{\tilde{\kappa}} \mathrm{V}(\Omega)u_{_{\delta}}|_+  &= \frac{\omega^2}{\tilde{\kappa}} \mathrm{V}(\Omega)u_{_{\delta}}|_- + \int_{\partial \Omega}\left.\frac{1}{\tilde{\rho}}\frac{\partial u_{_{\delta}}}{\partial \nu}\right|_- \dx{s} \\
&= \frac{\omega^2}{\tilde{\kappa}} \mathrm{V}(\Omega)u_{_{\delta}}|_- + \int_{ \Omega} -\frac{\omega^2}{\tilde{\kappa}} u_{_{\delta}} \dx{x}\\
&= \frac{\omega^2}{\tilde{\kappa}} \mathrm{V}(\Omega) (u_{_\delta}|_- - \widetilde{u}_{_\delta ,\Omega}),
\end{aligned}
\end{equation*}
where $\widetilde{u}_{_\delta ,\Omega}$ is the integral average of $u_{_{\delta}}$ over $\Omega$.
By the Poincar\'e –Wirtinger inequality, there exists a constant $C$ depending only on $\Omega$, such that for any function $u \in H^{1}(\Omega)$  satisfying
\begin{equation*}
\|u - \widetilde{u}_{\Omega}\|_{L^2(\Omega)} \leqslant C\|\nabla u\|_{L^2(\Omega)}.
\end{equation*}
Since $\nabla (u - \widetilde{u}_{\Omega}) = \nabla u$, it is clear that
$
\|u - \widetilde{u}_{\Omega}\|_{H^1(\Omega)}  \leqslant C\|\nabla u\|_{L^2(\Omega)}.
$
Therefore with the Poincar\'e –Wirtinger inequality, it holds that 
\begin{equation*}
\begin{aligned}
\left\|\int_{\partial \Omega} \left. \frac{1}{\rho}\frac{\partial u_{_{\delta}}}{\partial \nu}\right|_+\dx{s} + \frac{\omega^2}{\tilde{\kappa}} \mathrm{V}(\Omega)u_{_\delta}|_+ \right\|_{H^{\frac{1}{2}}(\partial \Omega)} &= 
\frac{\omega^2}{\tilde{\kappa}} \mathrm{V}(\Omega)\left\|  u_{_\delta}|_- - \widetilde{u}_{_\delta ,\Omega}\right\|_{H^{\frac{1}{2}}(\partial \Omega)} \\
&\leqslant \frac{\omega^2}{\tilde{\kappa}} \mathrm{V}(\Omega)\left\|  u_{_{\delta}} - \widetilde{u}_{_\delta ,\Omega}\right\|_{H^1( \Omega)} \\
&\leqslant C  \|\nabla u_{_{\delta}}\|_{L^2(\Omega)} \leqslant C\delta^{\frac{1}{2}}\|u^i\|_{H^1(B_R\backslash\overline{\Omega})},
\end{aligned}
\end{equation*}
where we have used the inequality \eqref{aprior2}.  Let $u_{zero} = u_{_{\delta}}-v$ where $v$ is the unique solution to the obstacle scattering \eqref{9-3}, we can see that $u_{zero}$ satisfies the system \eqref{d1} together with the boundary condition
$$ f := \int_{\partial \Omega}\frac{1}{\rho}\frac{\partial u_{_{\delta}}}{\partial \nu}\dx{s}  + \frac{\omega^2}{\tilde{\kappa}} \mathrm{V}(\Omega)u_{\delta} .$$ 
Hence it holds that
\begin{equation*}
\|u_{zero}\|_{H^1_{loc}(\bbr^N\backslash \overline{\Omega})} \leqslant C \left\|\int_{\partial \Omega} \frac{1}{\rho}\frac{\partial u_{_{\delta}}}{\partial \nu}  + \frac{\omega^2}{\tilde{\kappa}} \mathrm{V}(\Omega)u_{_\delta}\right\|_{H^{\frac{1}{2}}(\partial \Omega)} \leqslant C\delta^{\frac{1}{2}}\|u^i\|_{H^1(B_R\backslash\overline{\Omega})}.
\end{equation*}
We have proved $u_{_\delta} \rightarrow v $ in $H^1_{loc}(\bbr^N)$ as $\delta \rightarrow 0$. 

We further consider the case where both $\delta$ and $\tau$ take extreme values. When $\delta$ and $|\tau|$ are sufficiently small, we can get
\begin{equation*}
\|u_{_\delta}|_+\|_{H^{\frac{1}{2}}(\partial \Omega)} = \|u_{_\delta}|_-\|_{H^{\frac{1}{2}}(\partial \Omega)} \leqslant \|u_{_\delta}\|_{H^1(\Omega)} \leqslant C \sqrt{\delta + |\tau|} \|u^i\|_{H_{loc}^{1}(\bbr^N\backslash \overline{\Omega})},
\end{equation*}
where we have used the a-prior estimate \eqref{aprior3}. Let $u_{dzero} = u_{_\delta} - w$ and from Lemma \ref{dirichlet}, it is obvious that
\begin{equation*}
\|u_{dzero}\|_{H^1_{loc}(\bbr^N\backslash \overline{\Omega})} \leqslant C \|u_{_\delta}\|_{H^{\frac{1}{2}}(\partial \Omega)} \leqslant C \sqrt{\delta + |\tau|} \|u^i\|_{H_{loc}^{1}(\bbr^N\backslash \overline{\Omega})}.
\end{equation*}
We obtain the sound-soft obstacle scattering system \eqref{9-22}. 

For the obstacle scattering system \eqref{9-222}, we need to prove that $u_\delta$ converges to a constant on the boundary when $\delta \rightarrow 0$  and $|\tau| \rightarrow \infty$ .   This claim in fact is always correct  for any $|\tau|$ since the coefficient $C$ is independent of $|\tau|$ for the inequality
\begin{equation*}
\|\nabla u_{_\delta}\|_{L^2(\Omega)} \leqslant C \delta^{\frac{1}{2}} \|u^i\|_{H_{loc}^{1}(\bbr^N\backslash \overline{\Omega})}.
\end{equation*} 
Hence 
\begin{equation}\label{yy}
\left\|  u_{_\delta}|_+ - \widetilde{u}_{_\delta ,\Omega}\right\|_{H^{\frac{1}{2}}(\partial \Omega)} = \left\|  u_{_\delta}|_- - \widetilde{u}_{_\delta ,\Omega}\right\|_{H^{\frac{1}{2}}(\partial \Omega)} \rightarrow 0 \quad \text{\rm as} \quad \delta\rightarrow 0.
\end{equation}
It means that $u$ is constant on $\partial \Omega$ and equal the integral average $\widetilde{u}_{\delta ,\Omega}$ of $u_\delta$ over $\Omega$.
In addition, by using the Schwartz inequality and the estimate \eqref{udelta}, it holds that
\begin{equation}\label{xx} 
\begin{aligned}
\left|\int_{\partial \Omega} \left.\frac{1}{\rho}\frac{\partial u_{_{\delta}}}{\partial \nu}\right|_+\dx{s} \right| &= \left|\int_{\partial \Omega} \left.\frac{1}{\tilde{\rho}}\frac{\partial u_{_{\delta}}}{\partial \nu}\right|_-\dx{s} \right| = \left| \int_{\Omega} \frac{\omega^2}{\tilde{\kappa}} u_{_\delta} \dx{x}\right| \\
&\leqslant C |\tau|^{-1} \|u_{_\delta}\|_{L^2(\Omega)} \leqslant  C |\tau|^{-\frac{1}{2}} \|u^i\|_{H^1_{loc}(\bbr^N)},
\end{aligned}
\end{equation}
where C is some constant only depending on $\Omega, k, \alpha,\beta$.
Let $u_{dinf} = u_{_\delta} - t$ and we can find that $u_{dinf}$ satisfies \eqref{tss} with 
$$d = u_\delta - \widetilde{u}_{\delta ,\Omega} \coma a_0 = \int_{\partial \Omega} \frac{1}{\rho}\partial_\nu u_{_\delta}.$$
 Lemma \ref{33} together with   \eqref{yy} and \eqref{xx} shows that $u_{_\delta} \rightarrow t $ in $H^1_{loc}(\bbr^N)$ as $\delta \rightarrow 0$ and $|\tau| \rightarrow \infty$. 
 We have finished the whole proof.
\end{proof}

\section{The wave scattering from a ball}
In this section, we shall consider the three-dimensional acoustic scattering from a unit ball. The analysis of two-dimensional case can be treated in a similar manner with no much technical difficulty. We would prove that the inhomogeneous medium scattering would converge to the corresponding obstacle scattering in terms of $\delta$ and $\tau$, which reinforces our theoretical results in Section \ref{1dot2}.

Let $j_n(t)$ and $h^{(1)}_n(t)$ be, respectively, the spherical Bessel and Hankel functions of order $n \in \bbn$, and $Y_n^m$ be the spherical harmonics. For $\mb{x} \in \bbr^3$, the incident field can be rewritten in Fourier series as 
\begin{equation*}
u^i(r\hat{\mb{x}}) = \sum_{n=0}^\infty \sum_{m=-n}^{n}a_n^m j_n(kr)Y_n^m(\hat{\mb{x}}),
\end{equation*}
and we assume  $u^i \in H^1_{loc}(\bbr^3)$. 
The assumption of uniform convergence on compact subsets of $\bbr^3$ guarantees the summation boundedness of the coefficients, namely,
\begin{equation}\label{coeffiect}
    \sum_{n=0}^\infty \sum_{m=-n}^{n}|a_n^m|^2 \leqslant C,
\end{equation}
for some positive constant $C$.
With the above series representation of $u^i$, the solution to transmission problem $\eqref{9-1}$ for a unit ball can be rewritten as 
\begin{equation}\label{presentation}
u_{_{\delta}}(r\hat{\mb{x}}) =
\begin{cases}
 \displaystyle \sum_{n=0}^\infty\sum_{m=-n}^{n}b_n^m j_n(k_br)Y_n^m(\hat{\mb{x}})  \coma x \in B_1 , \medskip\\
 \displaystyle \sum_{n=0}^\infty\sum_{m=-n}^{n}c_n^m h_n^{(1)}(kr)Y_n^m(\hat{\mb{x}}) + u^i(r\hat{\mb{x}}) \coma x \in \bbr^3\backslash \overline{B_1}, \medskip
\end{cases}
\end{equation}
where $k_b = k\sqrt{\delta / \tau}$. Applying the third line transmission boundary condition in \eqref{9-1} and $u_{_\delta}$ represented in \eqref{presentation}, we have
\begin{equation*}
\begin{aligned}
 b_n^m j_n(k_b) &=  a_n^m j_n(k) + c_n^m  h^{(1)}_n(k) \coma
b_n^m k_b j'_n(k_b) = \delta \left(a_n^m k j'_n(k) + c_n^m k {h_n^{(1)}}'(k) \right).
\end{aligned}
\end{equation*}
Solving the above equations yields
\begin{equation*}
\begin{aligned}
b_n^m = \frac{j_n(k){h_n^{(1)}}'(k) - j'_n(k)h_n^{(1)}(k)}{{h_n^{(1)}}'(k)j_n(k_b) - \frac{1}{\sqrt{\delta \tau}}h_n^{(1)}(k)j'_n(k_b)} a_n^m,\\
c_n^m = \frac{j_n(k)j'_n(k_b) - \sqrt{\delta \tau} j'_n(k)j_n(k_b)}{\sqrt{\delta \tau}{h_n^{(1)}}'(k)j_n(k_b) - h_n^{(1)}(k)j'_n(k_b)}a_n^m.
\end{aligned}
\end{equation*}
Similarly, For obstacle scattering problems \eqref{9-3},  \eqref{9-222}, \eqref{9-22}, \eqref{9-2}, the total fields illuminated by the incident field $u^i$ in series representations are
\begin{align}
v(r \hat{\mb{x}}) =& - \sum_{n=1}^\infty\sum_{m=-n}^{n}a_n^m\frac{j_n(k)}{{h}_n^{(1)}(k)}h_n^{(1)}(kr)Y_n^m(\hat{\mb{x}}) \nonumber  \\
&- a_0^0\frac{3\tau  j'_0(k)+kj_0(k)}{3 \tau {h_0^{(1)}}'(k)+ k h_0^{(1)}(k)}h_0^{(1)}(kr)Y_0^0(\hat{\mb{x}}) + u^i(r\hat{\mb{x}}), \label{v} \\
\psi(r \hat{\mb{x}}) =&  -  \sum_{n=0}^\infty\sum_{m=-n}^{n}a_n^m\frac{j_n(k)}{{h}_n^{(1)}(k)}h_n^{(1)}(kr)Y_n^m(\hat{\mb{x}}) + u^i(r\hat{\mb{x}}) \coma \psi = w(r \hat{\mb{x}}), t(r \hat{\mb{x}})\label{w} , \\
u(r \hat{\mb{x}}) =& -\sum_{n=0}^\infty \sum_{n=-m}^m a_n^m \frac{j'_n(k)}{{h^{(1)}_n}'(k)} h^{(1)}_n(kr) Y_n^m(\hat{\mb{x}}) +u^i(r\hat{\mb{x}}), \label{u} 
\end{align}
It is noted that $t(r\hat{\mb{x}})$ shares the same mathematical representation of solution with $w(r \hat{\mb{x}})$ for the special obstacle with the unit ball shape.

The solution to the Helmholtz equation for the radial geometry can give an explicit presentation as the series expansion,  which is convenient to validate the convergence results. We shall prove the convergence of $u_{_\delta}$  when the material parameters degenerate to zero and infinity.

\begin{proposition}
Consider the inhomogeneous medium scattering \eqref{9-1} and assume $u^i \in H_{loc}^1({\bbr^3})$. Then the total field $u_{_\delta}$ illuminated by the incident field $u^i$ on a unit ball has the following convergences:
\begin{align*}
&u_{_\delta} \rightarrow v  \; \text{\rm in} \; H_{loc}^1({\bbr^3\backslash \overline{B_1}}) \; \text{\rm as} \; \delta \rightarrow 0; \hspace{18pt}
u_{_\delta} \rightarrow t \; \text{\rm in} \; H_{loc}^1({\bbr^3\backslash \overline{B_1}}) \; \text{\rm as} \; \delta \rightarrow 0 \; \text{\rm and} \; |\tau| \rightarrow \infty;\\
&u_{_\delta} \rightarrow u  \; \text{\rm in} \; H_{loc}^1({\bbr^3 \backslash \overline{B_1}}) \; \text{\rm as} \; \delta \rightarrow \infty ; \quad  
u_{_\delta} \rightarrow w  \; \text{\rm in} \; H_{loc}^1({\bbr^3\backslash \overline{B_1}}) \; \text{\rm as} \; \delta \rightarrow 0 \; \text{\rm and} \; |\tau| \rightarrow 0;
\end{align*} 
where $v,t,w,u$ are the unique solutions to \eqref{9-3},  \eqref{9-222},\eqref{9-22},\eqref{9-2} respectively.
\end{proposition}

\begin{proof}
From the solutions presented in \eqref{presentation} and \eqref{u}, we can obtain when $x \in \bbr^3 \backslash \overline{B_1}$,
\begin{equation}\label{quu}
\begin{aligned}
&u_{_{\delta}}(r\hat{\mb{x}}) - u(r \hat{\mb{x}}) \\
=&  \sum_{n=0}^\infty\sum_{m=-n}^{n}(c_n^m+ a_n^m \frac{j'_n(k)}{{h^{(1)}_n}'(k)}) h_n^{(1)}(kr)Y_n^m(\hat{\mb{x}}) \\
=&   \sum_{n=0}^\infty\sum_{m=-n}^{n} a_n^m \frac{j_n(k){h_n^{(1)}}'(k) - j'_n(k)h_n^{(1)}(k)}{{\sqrt{\delta \tau}}{h_n^{(1)}}'(k)j_n(k_b) -h_n^{(1)}(k)j'_n(k_b)} \frac{j_n'(k_b)}{{h_n^{(1)}}'(k)} h_n^{(1)}(kr)Y_n^m(\hat{\mb{x}}) \\
:=& \sum_{n=0}^\infty\sum_{m=-n}^{n} q_n a_n^m h_n^{(1)}(kr)Y_n^m(\hat{\mb{x}}) \coma 
\end{aligned}
\end{equation}

where $q_n$ is defined as
\begin{equation}\label{qn}
q_n = \delta^{-\frac{1}{2}} \frac{j_n(k){h_n^{(1)}}'(k) - j'_n(k)h_n^{(1)}(k)}{{\sqrt{\tau}}{h_n^{(1)}}'(k)j_n(k_b) - \delta^{-\frac{1}{2}} h_n^{(1)}(k)j'_n(k_b)} \frac{j_n'(k_b)}{{h_n^{(1)}}'(k)}.
\end{equation}

In virtue of the Wronskian  
\begin{equation}\label{wronskian}
j_n(x)y_n'(x) - j'_n(x)y_n(x) = \frac{1}{x^2},
\end{equation}
 we have
\begin{equation} \label{qnw}
\begin{aligned}
q_n &= \mi\delta^{-\frac{1}{2}} \frac{j_n(k)y'_n(k) - j'_n(k)y_n(k)}{{\sqrt{\tau}}{h_n^{(1)}}'(k)j_n(k_b) - \delta^{-\frac{1}{2}} h_n^{(1)}(k)j'_n(k_b)} \frac{j_n'(k_b)}{{h_n^{(1)}}'(k)}\\
&=\frac{\mi \delta^{-\frac{1}{2}}}{k^2} \frac{1}{{h_n^{(1)}}'(k)({\sqrt{\tau}}{Q(n,k_b)h_n^{(1)}}'(k) - \delta^{-\frac{1}{2}} h_n^{(1)}(k))} ,
\end{aligned}
\end{equation}
where 
$$Q(n,k_b) = \frac{j_n(k_b)}{j'_n(k_b)}.$$
Since $ \delta \rightarrow \infty$, then $ \Re(k_b) , \Im(k_b) \rightarrow \infty.$
 From Lemma \ref{b2}, we can get
\begin{equation*}
Q(n,k_b)  \sim e^{-\mi \frac{\pi}{2}}.
\end{equation*}
Hence, we can readily obtain
\begin{equation*}
    q_n \sim C \delta^{-\frac{1}{2}} \coma n \in\bbn \coma \delta \rightarrow +\infty,
\end{equation*}
where $C$ is some constant independent of $\delta$ and $n \in \bbn$.

In order to obtain the uniform convergence for $ u_{_\delta}-u$, we need the asymptotic expansions of $j_n(x),h_n^{(1)}(x)$  for large orders. 
Combining  \eqref{qnw} and Lemma \ref{b1}, we obtain by direct calculations
\begin{equation*}
\begin{aligned}
    q_n \frac{h_n^{(1)}(kr)}{j_n(kr)} &= \frac{\mi \delta^{-\frac{1}{2}}}{k^2}\frac{h_n^{(1)}(kr)}{{h_n^{(1)}}'(k)j_n(kr)}  \frac{1}{{\sqrt{\tau}}{Q(n,k_b)h_n^{(1)}}'(k) - \delta^{-\frac{1}{2}} h_n^{(1)}(k)} \\
    &\sim \frac{2n+1}{n+1}\frac{k}{(n+1)\sqrt{\tau}\mi + k \delta^{-\frac{1}{2}}}\frac{1}{r^{2n+1}} \delta^{-\frac{1}{2}},
\end{aligned}
\end{equation*}
when $n$ is sufficiently large. Since $r>1$, there exists a constant $C$ independent of $n,m$ and $\delta$ such that
\begin{equation*}
    |(u_{_{\delta}}(r\hat{\mb{x}}) - u(r \hat{\mb{x}}))_n^m| \leqslant C \delta^{-\frac{1}{2}} |(u^i(r\hat{\mb{x}}))_n^m|,
\end{equation*} 
where $(\cdot)_n^m$ is the corresponding $nm$-th component in the series \eqref{quu}.
By the absolute homogeneity of norm, we can get $u_{_{\delta}} \rightarrow u$ as $\delta \rightarrow \infty$.

Similarly, we obtain from \eqref{presentation} and \eqref{v} 
\begin{equation}\label{qvv}
\begin{aligned}
&u_{_{\delta}}(r\hat{\mb{x}}) - v(r \hat{\mb{x}}) \\
=&  \sum_{n=0}^\infty\sum_{m=-n}^{n}(c_n^m+ a_n^m \frac{j_n(k)}{{h^{(1)}_n}(k)}) h_n^{(1)}(kr)Y_n^m(\hat{\mb{x}}) \\
=& a_0^0\frac{(j_0(k){h_0^{(1)}}'(k)-j'_0(k)h_0^{(1)}(k))(3\tau j'_0(k_b)+k\sqrt{\delta \tau}j_0(k_b))}{(\sqrt{\delta \tau}{h_0^{(1)}}'(k)j_0(k_b) - h_0^{(1)}(k)j'_0(k_b))(3 \tau {h_0^{(1)}}'(k)+ k h_0^{(1)}(k))}h_0^{(1)}(kr) Y_0^0(\hat{\mb{x}})\\
&+ \sum_{n=1}^\infty\sum_{m=-n}^{n} a_n^m  \frac{\sqrt{\delta\tau}(j_n(k){(h_n^{(1)}}'(k) - j'_n(k)h_n^{(1)}(k))}{{\sqrt{\delta \tau}}{h_n^{(1)}}'(k)j_n(k_b) -h_n^{(1)}(k)j'_n(k_b)} \frac{j_n(k_b)}{{h_n^{(1)}}(k)} h_n^{(1)}(kr)Y_n^m(\hat{\mb{x}}) \\
:=& \sum_{n=0}^\infty\sum_{m=-n}^{n} p_n a_n^m h_n^{(1)}(kr)Y_n^m(\hat{\mb{x}})  \coma r >1 ,
\end{aligned}
\end{equation}
where
\begin{equation*}
 \begin{aligned}
p_0 &= \sqrt{\tau}\frac{(j_0(k){h_0^{(1)}}'(k)-j'_0(k)h_0^{(1)}(k))(3\sqrt{\tau} j'_0(k_b)+k\sqrt{\delta}j_0(k_b))}{(\sqrt{\delta \tau}{h_0^{(1)}}'(k)j_0(k_b) - h_0^{(1)}(k)j'_0(k_b))(3 \tau {h_0^{(1)}}'(k)+ k h_0^{(1)}(k))} \coma\\
p_n &= \sqrt{\delta\tau} \frac{j_n(k){h_n^{(1)}}'(k) - j'_n(k)h_n^{(1)}(k)}{{\sqrt{\delta \tau}}{h_n^{(1)}}'(k)j_n(k_b) -h_n^{(1)}(k)j'_n(k_b)} \frac{j_n(k_b)}{{h_n^{(1)}}(k)} \coma n \geqslant 1.
\end{aligned}   
\end{equation*}
We use the Wronskian \eqref{wronskian} again to rewrite $p_n$ as
\begin{equation*}
\begin{aligned}
p_n &= \frac{\mi\sqrt{\delta\tau}}{k^2} \frac{j_n(k_b)}{{h_n^{(1)}}(k)({\sqrt{\delta \tau}}{h_n^{(1)}}'(k)j_n(k_b) -h_n^{(1)}(k)j'_n(k_b))} \\
&= \frac{\mi\sqrt{\delta\tau}}{k^2} \frac{1}{{h_n^{(1)}}(k)({\sqrt{\delta \tau}}{h_n^{(1)}}'(k) -h_n^{(1)}(k)Q(n,k_b))}
\end{aligned}
\end{equation*}
When $\delta \rightarrow 0$, it is clear that $k_b =k(\delta/\tau)^{1/2} \rightarrow 0$ and hence $Q(n,k_b) \sim n/k_b$. From the Lemma \ref{b2}, it holds that
\begin{equation*}
p_n \sim \frac{\mi\sqrt{\delta\tau}}{k^2} \frac{1}{{h_n^{(1)}}(k)({\sqrt{\delta \tau}}{h_n^{(1)}}'(k) - nh_n^{(1)}(k)/k_b)} \leqslant C \delta^{\frac{1}{2}}  \coma \delta \rightarrow 0 \coma n \in \bbn .
\end{equation*}
When $n$ is sufficiently large, we can readily from Lemma \ref{b1} get 
   \begin{equation*} 
    p_n \frac{h_n^{(1)}(kr)}{j_n(kr)} \sim \frac{2n+1}{n}\frac{1}{r^{2n+1}}\frac{n}{(n+1)\delta^{\frac{1}{2}} +n \tau \delta^{-\frac{1}{2}}} \delta^{\frac{1}{2}} \rightarrow 0.
\end{equation*} 
Finally, we can get
\begin{equation}\label{vpn}
    |(u_{_{\delta}}(r\hat{\mb{x}}) - v(r \hat{\mb{x}}))_n^m| \leqslant C \delta^{\frac{1}{2}} |(u^i(r\hat{\mb{x}}))_n^m| \coma n \geqslant 1.
\end{equation} 
where $(\cdot)_n^m$ is the corresponding $nm$-th component in the series \eqref{qvv}.
When $n=0$, we can rewrite $p_0$ as
\begin{equation*}
\begin{aligned}
p_0 &= \frac{\mi}{k^2}\frac{3\tau j'_0(k_b)/j_0(k_b)+ k\sqrt{\delta \tau}}{(\sqrt{\delta \tau}{h_0^{(1)}}'(k) - h_0^{(1)}(k)j'_0(k_b)/j_0(k_b))(3 \tau {h_0^{(1)}}'(k)+ k h_0^{(1)}(k))}
\end{aligned}
\end{equation*}
Since $j_0(x) = \sin x/ x$, straight calculations gives
\begin{equation*}
\frac{j_0'(k_b)}{j_0(k_b)}  = \frac{1}{\tan k_b} - \frac{1}{k_b} \sim - \frac{1}{3}k_b -\frac{1}{45}k_b^3 + \mathcal{O}(k_b^5) \coma k_b \rightarrow 0.
\end{equation*}
Then,
\begin{equation}\label{vp0}
p_0 = -\frac{\mi  k_b^2}{45k}\frac{\tau + \mathcal{O}(k_b^2)}{(\tau h_0^{(1)}(k)+ kh_0^{(1)}(k)(1+\frac{1}{15}k_b^2+\mathcal{O}(k_b^4)))(3 \tau {h_0^{(1)}}'(k)+ k h_0^{(1)}(k))} \sim k_b^2.
\end{equation}
Combining \eqref{vpn} and \eqref{vp0}, we can conclude that
 $u_{_{\delta}} \rightarrow v$ in $H_{loc}^1(\bbr^3)$ as $\delta \rightarrow 0$.
 
We now prove the remaining two convergences to the solution presented in \eqref{w}. Direct calculation gives
\begin{align*}
&u_{_{\delta}}(r \hat{\mb{x}}) - \psi(r \hat{\mb{x}})  \\
=&  \sum_{n=0}^\infty\sum_{m=-n}^{n}(c_n^m+ a_n^m \frac{j_n(k)}{{h^{(1)}_n}(k)}) h_n^{(1)}(kr)Y_n^m(\hat{\mb{x}}) \\
=&\sum_{n=0}^\infty\sum_{m=-n}^{n} a_n^m  \frac{\sqrt{\delta\tau} (j_n(k){h_n^{(1)}}'(k) - j'_n(k)h_n^{(1)}(k))}{{\sqrt{\delta \tau}}{h_n^{(1)}}'(k)j_n(k_b) -h_n^{(1)}(k)j'_n(k_b)} \frac{j_n(k_b)}{{h_n^{(1)}}(k)} h_n^{(1)}(kr)Y_n^m(\hat{\mb{x}}) \\
=&\sum_{n=0}^\infty\sum_{m=-n}^{n} a_n^m \frac{\mi}{k^2} \frac{\sqrt{\delta\tau}}{{\sqrt{\delta \tau}}{h_n^{(1)}}'(k)j_n(k_b) -h_n^{(1)}(k)j'_n(k_b)} \frac{j_n(k_b)}{{h_n^{(1)}}(k)} h_n^{(1)}(kr)Y_n^m(\hat{\mb{x}})  \\
:=& \sum_{n=0}^\infty\sum_{m=-n}^{n} l_n a_n^m h_n^{(1)}(kr)Y_n^m(\hat{\mb{x}})  \coma
\end{align*}
where
\begin{equation*}
l_n = \frac{\mi}{k^2}\frac{1}{h_n^{(1)}(k)}\frac{1}{{h_n^{(1)}}'(k) - {h_n^{(1)}}(k)Q(n,k_b) (\delta\tau)^{-\frac{1}{2}}}.
\end{equation*}
It is emphasized that $j_n(k_b)$ is always nonzero since the zeros of $j_n$ are all real and $k_b$ is a complex number with a positive imaginary part. We shall consider $Q(n,k_b)$ in two extreme cases. 
When $\delta \rightarrow 0$ and $|\tau| \rightarrow \infty$,  it is clear that $k_b \rightarrow 0$. Hence 
\begin{equation*}
l_n \sim  \frac{\mi}{k}\frac{1}{h_n^{(1)}(k)}\frac{1}{{kh_n^{(1)}}'(k) - {h_n^{(1)}}(k)n \delta} \sim \delta \rightarrow 0 \coma n \in \bbn.
\end{equation*}
When $n$ is sufficiently large, it holds that
\begin{equation*}
l_n \frac{h_n^{(1)}(kr)}{j_n(kr)}  \sim \frac{2n+1}{n+1}\frac{k}{(n+1)\sqrt{\tau}\mi + k \delta^{-\frac{1}{2}}}\frac{1}{r^{2n+1}} \delta^{-\frac{1}{2}} \rightarrow 0.
\end{equation*}
The uniform boundedness shows
\begin{equation*} 
    |(u_{_{\delta}}(r\hat{\mb{x}}) - \psi(r \hat{\mb{x}}))_n^m| \leqslant C \delta^{\frac{1}{2}} |(u^i(r\hat{\mb{x}}))_n^m|,
\end{equation*} 
and hence  $u_{_{\delta}} \rightarrow t$ in $H_{loc}^1(\bbr^3)$ as $\delta \rightarrow 0$ and $|\tau| \rightarrow \infty$.

When $\delta \rightarrow 0$ and $|\tau| \rightarrow 0$, there is no restriction about the asymptotic behaviors of $k_b$. We need to analyze $k_b$ for different extreme values. First , it is clear from Lemma \ref{b2} that
\begin{equation*}
Q(n,k_b) (\delta \tau)^{-\frac{1}{2}} \sim
\begin{cases}
\displaystyle\frac{n}{k}\delta^{-1} \coma k_b \rightarrow 0 \coma n \in \bbn ,\\
\displaystyle (\delta \tau)^{-\frac{1}{2}}   \coma k_b \sim 1 \coma n \in \bbn, \\
\displaystyle -\mi (\delta \tau)^{-\frac{1}{2}}\coma k_b \rightarrow \infty \coma n \in \bbn.
\end{cases}
\end{equation*} 
Then we can obtain
\begin{equation*}
l_n \sim  \max\{(\delta\tau)^{\frac{1}{2}},\delta\} \rightarrow 0.
\end{equation*}
When $n$ is sufficiently large, it holds that
\begin{equation*}
l_n \frac{h_n^{(1)}(kr)}{j_n(kr)}  \sim \frac{2n+1}{n+1}\frac{k}{(n+1)\sqrt{\tau}\mi + k \delta^{-\frac{1}{2}}}\frac{1}{r^{2n+1}}  \max\{(\delta\tau)^{\frac{1}{2}},\delta\} \rightarrow 0.
\end{equation*}
hence  $u_{_{\delta}} \rightarrow w$ in $H_{loc}^1(\bbr^3)$ as $\delta \rightarrow 0$ and $|\tau| \rightarrow 0$. We have finished the proof.
\end{proof}

The far-field pattern describes the scattering amplitude in the infinity, which contains all the information of the scattering field in every direction. The following corollary is a straightforward corollary for the far-field patterns.
\begin{corollary}
Let $\mathcal{F}_\delta,\mathcal{F}_v,\mathcal{F}_t,\mathcal{F}_w,\mathcal{F}_u$ be the far field patterns of the systems \eqref{9-1}, \eqref{9-3},  \eqref{9-222},\eqref{9-22},\eqref{9-2} respectively. then the following sharp estimates hold:
\begin{align*}
&\mathcal{F}_\delta \rightarrow  \mathcal{F}_u  \; \textit{\rm in}  \; L^2(\mathbb{S}^2) \; \textit{\rm  as} \; \delta \rightarrow \infty \; ; \quad  
\mathcal{F}_\delta \rightarrow  \mathcal{F}_w  \; \textit{\rm  in}  \; L^2(\mathbb{S}^2) \; \textit{\rm as} \; \delta \rightarrow 0 \; \text{\rm and} \; |\tau| \rightarrow 0.   \\
&\mathcal{F}_\delta \rightarrow  \mathcal{F}_v  \; \textit{\rm  in}  \; L^2(\mathbb{S}^2) \; \textit{\rm  as} \; \delta \rightarrow 0 \; ; \hspace{17pt}
\mathcal{F}_\delta \rightarrow  \mathcal{F}_q  \; \textit{ \rm  in}  \; L^2(\mathbb{S}^2) \; \textit{\rm  as} \; \delta \rightarrow 0 \; \text{\rm and} \; |\tau| \rightarrow \infty.
\end{align*}

\end{corollary}
\begin{proof}
From the representations of the scattering fields, we have
\begin{align*}
\mathcal{F}_\delta (\hat{\mb{x}}) - \mathcal{F}_u (\hat{\mb{x}})  &= \frac{1}{k}\sum_{n=0}^\infty \frac{1}{\mi^{n+1}}\sum_{m=-n}^n (c_n^m + \frac{j'_n(k)}{{h_n^{(1)}}'(k)}a_n^m)  Y_n^m(\hat{\mb{x}}) \\
&= \frac{1}{k}\sum_{n=0}^\infty \frac{1}{\mi^{n+1}}\sum_{m=-n}^n q_n a_n^m  Y_n^m(\hat{\mb{x}}),
\end{align*}
where $q_n$ are defined in \eqref{qn}. By the orthogonality of $Y_n^m$ and the boundedness in \eqref{coeffiect}, we can get
\begin{align*}
    \|\mathcal{F}_\delta (\hat{\mb{x}}) - \mathcal{F}_u (\hat{\mb{x}})\|_{L^2(\mathcal{S}^2)}^2 = \frac{1}{k^2}\sum_{n=0}^\infty q_n^2 \sum_{m=-n}^n (a_n^m)^2 \sim \frac{1}{k^2}\delta^{-1}  \rightarrow 0 \quad \textit{\rm as} \; \delta \rightarrow  \infty.
\end{align*}
So $\mathcal{F}_\delta \rightarrow  \mathcal{F}_u$ in  $L^2(\mathbb{S}^2)$. Similarly, we can prove the remaining three convergences in $L^2(\mathbb{S}^2)$.
\end{proof}


\appendix
\section{}
\subsection{Layer potentials}\label{layer}
We shall make use layer-potential techniques to prove the wellposedness results in the previous sections. To that end, we introduce some necessary notations and results on the layer potential operators. Let $\Phi$ be the fundamental solution of the  PDO $\Delta+k^2$:
\begin{equation*}
\Phi (\mb{x}, \mb{y})= 
\begin{cases}
\displaystyle H_0^{(1)}(k|\mb{x}-\mb{y}|) \coma N=2 , \\
\displaystyle -\frac{e^{\mathrm{i} k |\mb{x}-\mb{y}|}}{4\pi|\mb{x}-\mb{y}|} \coma N=3.
\end{cases}
\end{equation*}
Let $B \subset \bbr^N$ be a bounded $C^{2,\alpha}$ domain, the single- and double-layer potential operators are bounded from $L^2(\partial B)$ into $H^1(\bbr^N /\partial B)$ and given by
\begin{align*}
\mathbb{S}_B [\psi] (\mb{x})  &= \int_{\partial B} \Phi (\mb{x},\mb{y}) \psi(\mb{y}) \dx{s_\mb{y}} \coma \mb{x} \in \bbr^N/\partial B, \\
\mathbb{D}_B [\psi] (\mb{x}) &= \int_{\partial B} \frac{\partial \Phi (\mb{x},\mb{y})}{\partial \nu_\mb{y}} \psi(\mb{y}) \dx{s_\mb{y}} \coma \mb{x} \in \bbr^N/\partial B,
\end{align*}
where $\psi$ is a density function defined on the boundary $\partial B$. Here $\nu_{\mb{x}}\in\mathbb{S}^{N-1}$ signifies the exterior unit normal vector to the boundary of the concerned domain $B$ at $\mb{x}$. 

For the single- and double-layer potential on the boundary $\partial B$, we need more regularity. We omit the subscript $B$ of these operators in the sequel for simplicity and it would be clear for the context.
Let $S,K$ respectively denote the single- and double-layer potential operators on the boundary
, given by
\begin{align*}
S [\psi] (\mb{x})  &= \int_{\partial B} \Phi (\mb{x},\mb{y}) \psi(\mb{y}) \dx{s_\mb{y}} \coma \mb{x} \in\partial B, \\
K [\psi] (\mb{x}) &= \int_{\partial B} \frac{\partial \Phi (\mb{x},\mb{y})}{\partial \nu_\mb{y}} \psi(\mb{y}) \dx{s_\mb{y}} \coma \mb{x} \in \partial B,
\end{align*}
and the normal derivative operators $K^*$ and $T$ are given by
\begin{align*}
K^* [\psi] (\mb{x})  &= \int_{\partial B} \frac{\partial \Phi (\mb{x},\mb{y})}{\partial \nu_\mb{x}} \psi(\mb{y}) \dx{s_\mb{y}} \coma \mb{x} \in \partial B,\\
T [\psi] (\mb{x}) &= \frac{\partial}{\partial \nu_{\mb{x}}}\int_{\partial B} \frac{\partial \Phi (\mb{x},\mb{y})}{\partial \nu_\mb{y}} \psi(\mb{y}) \dx{s_\mb{y}} \coma \mb{x} \in \partial B .
\end{align*}
The behaviors of potential operators across the boundary are described by the jump relations. 
It is known that the single-layer potential is continuous across $\partial B$ and the normal derivative satisfies the trace formula
\begin{equation*}
\left. \frac{\partial}{\partial \nu}\mathbb{S}_B [\psi] (\mb{x})\right|_{\pm} = \left(\pm\frac{I}{2} + K^*\right) [\psi](\mb{x}) \coma \mb{x}\in \partial B,
\end{equation*}
where the subscripts $\pm$ indicate the limits from outside and inside of the domain $B$ respectively. The double-layer potential satisfies the trace formula
\begin{equation*}
\left. \mathbb{D}_B [\psi] (\mb{x}) \right|_{\pm} = \left(\mp \frac{I}{2} +K \right) [\psi] (\mb{x}) \coma \mb{x}\in \partial B,
\end{equation*}
and 
\begin{equation*}
\left. \frac{\partial}{\partial \nu}\mathbb{D}_B [\psi] (\mb{x}) \right|_+ - \left. \frac{\partial}{\partial \nu}\mathbb{D}_B [\psi] (\mb{x}) \right|_- = 0  \coma \mb{x}\in \partial B.
\end{equation*}

We recall the mapping properties of potential operators.
\begin{lemma}
Let $\partial B$ be of class $C^{2,\alpha}$. Then the single-layer potential $\mathbb{S}$ defines a bounded linear operator from $H^{-\frac{1}{2}}(\partial B)$ into $H^1_{loc}(\bbr^N \backslash \overline{B})$. the double-layer potential $\mathbb{D}$ defines a bounded linear operator from $H^{\frac{1}{2}}(\partial B)$ into $H^1_{loc}(\bbr^N \backslash \overline{B})$. Furthermore, the surface potential operator $S$ is bounded from $L^2(\partial B)$ into $H^1(\partial B)$. The operator $K$ and $K^*$ are bounded from $H^{-\frac{1}{2}}(\partial B)$ into $H^{\frac{1}{2}}(\partial B)$. The operator $T$ is bounded from $H^{\frac{1}{2}}(\partial B)$ into $H^{-\frac{1}{2}}(\partial B)$.
\end{lemma}

\subsection{Special functions}
The spherical Bessel functions and Hankel functions are the solutions of the spherical Bessle differential equation:
\begin{equation*}
x^2 \frac{d^2y}{dx^2} +2x\frac{dy}{dx} + (x^2 -n(n+1))y = 0.
\end{equation*}
By direct calculations, the solutions of spherical Bessle differential equation can be represented as
\begin{align*}
j_n(x) &= \sum_{m=0}^\infty \frac{(-1)^mx^{n+2m}}{2^mm!1 \cdot 3 \cdots (2n+2m+1)}, \\
y_n(x) &= -\frac{(2n)!}{2^nn!}\sum_{m=0}^\infty \frac{(-1)^m x^{2m-n-1}}{2^mm!(-2n+1)(-2n+3)\cdots(-2n+2m-1)},
\end{align*} 
where $n = 0,1,2,\cdots$.
The functions $j_n$ and $y_n$ are called spherical Bessel functions and spherical Neumann functions. The linear combinations
\begin{equation*}
h_n^{(1,2)} := j_n \pm \mi y_n
\end{equation*}
are known as spherical Hankel functions of the first and second kind of order $n$.

Let us the recall some asymptotic formulas of Spherical Bessel and Hankel functions.
\begin{lemma}\label{b1}
when $n$ is sufficiently large, it holds that
\begin{equation*}
j_n(x) = \frac{x^n}{(2n+1)!!}(1+ \mathcal{O}(\frac{1}{n})) \coma h_n^{(1)}(x) = \frac{(2n-1)!!}{\mi x^{n+1}}(1+ \mathcal{O}(\frac{1}{n})) ,
\end{equation*}
and 
\begin{equation*}
  j'_n(x) = \frac{nx^{n-1}}{(2n+1)!!}(1+ \mathcal{O}(\frac{1}{n})) \coma {h_n^{(1)}}'(x) = -\frac{(n+1)(2n-1)!!}{\mi x^{n+2}}(1+ \mathcal{O}(\frac{1}{n})).
\end{equation*}
\end{lemma}
\begin{proof}
With the aid of the recurrence relations 
\begin{equation}\label{recu}
x\varphi'_n(x) = n\varphi_n(x) - x\varphi_{n+1}(x) \coma \varphi_n (x)= j_n(x),h_n^{(1)}(x), 
\end{equation}
we can get
\begin{align*}
  j'_n(x) &= \frac{n}{x}j_{n}(x) - j_{n+1}(x) = \frac{nx^{n-1}}{(2n+1)!!}\left(1 - \frac{x^2}{n(2n+3)} +\mathcal{O}(\frac{1}{n}) \right) \\
  &= \frac{nx^{n-1}}{(2n+1)!!}(1+ \mathcal{O}(\frac{1}{n}))
\end{align*}
and 
\begin{align*}
 {h_n^{(1)}}'(x) &= \frac{n}{x} h_n^{(1)}(x) - h_{n+1}^{(1)}(x) =\frac{(2n+1)!!}{\mi x^{n+2}}\left(\frac{n}{2n+1}-1 +\mathcal{O}(\frac{1}{n}) \right) \\
  &= -\frac{(n+1)(2n-1)!!}{\mi x^{n+2}}(1+ \mathcal{O}(\frac{1}{n})).
\end{align*}
We have finished the proof.
\end{proof}

\begin{lemma}\label{b2}
When $x \rightarrow 0$, it holds that
\begin{equation*}
\frac{j'_n(x)}{j_n(x)} = \frac{n}{x} \left(1 +\mathcal{O}(x^2) \right)  \sim \frac{n}{x}.
\end{equation*}
When $\Re x \rightarrow \infty$ and $\Im x \rightarrow \infty$, it holds that
\begin{equation*}
\frac{j'_n(x)}{j_n(x)} =  e^{-\frac{\pi}{2}\mi } (1 + \mathcal{O}(|x|^{-1})) \sim e^{-\mi \frac{\pi}{2}}.
\end{equation*}
\end{lemma}

\begin{proof}
When $x \rightarrow 0$, the the recurrence relation \eqref{recu} gives
\begin{equation*}
\frac{j'_n(x)}{j_n(x)}  = \frac{n}{x} - \frac{j_{n+1}(x)}{j_n(x)} = \frac{n}{x} - \frac{2n+1}{2n+3}x (1 + \mathcal{O}(x^2)) =  \frac{n}{x} \left(1 +\mathcal{O}(x^2) \right) .
\end{equation*}
For the large arguments, $j_n(z)$ have the following asymptotic expansions
\begin{align*}
j_n(x) &= \frac{1}{x} \left(\cos (x- \frac{n\pi}{2}- \frac{\pi}{2}) + e^{|\Im (x)|}\mathcal{O}(|x|^{-1}) \right)  \\
&= \frac{1}{2x} \left(e^{\mi(x- \frac{n\pi}{2}- \frac{\pi}{2})} + e^{-\mi(x- \frac{n\pi}{2}- \frac{\pi}{2})}+ e^{|\Im (x)|}\mathcal{O}(|x|^{-1}) \right) \\
&=\frac{1}{2x} e^{|\Im x|}\left( e^{\mi(-\Re x + \frac{n\pi}{2} + \frac{\pi}{2})}+ \mathcal{O}(|x|^{-1}) \right)\coma |\arg x|<\pi ,
\end{align*}
where we have used the Euler's formula. Using the recurrence formula \eqref{recu}  again yields
\begin{equation*}
\frac{j'_n(x)}{j_n(x)}  = \frac{n}{x} - \frac{j_{n+1}(x)}{j_n(x)} = \frac{n}{x} - e^{\mi \frac{\pi}{2}}(1+ \mathcal{O}(|x|^{-1}))  = e^{- \frac{\pi}{2}\mi} (1 + \mathcal{O}(|x|^{-1})).
\end{equation*}
We have finished the proof.
\end{proof}

\section*{Acknowledgments}

The work was supported by NSFC/RGC Joint Research Scheme (project N\_{CityU}\linebreak[1] 101/21), ANR/RGC Joint Research Scheme (project A\_CityU203/19) and the Hong Kong RGC General Research Funds (projects 11311122, 11304224 and 11300821).

\bibliographystyle{abbrv}
\bibliography{referenceasy.bib}

\begin{thebibliography}{10}

\bibitem{ammari2020close}
H.~Ammari, B.~Davies, and S.~Yu.
\newblock Close-to-touching acoustic subwavelength resonators: eigenfrequency
  separation and gradient blow-up.
\newblock {\em Multiscale Modeling $\&$ Simulation}, 18(3):1299--1317, 2020.

\bibitem{ammari2018minnaert}
H.~Ammari, B.~Fitzpatrick, D.~Gontier, H.~Lee, and H.~Zhang.
\newblock Minnaert resonances for acoustic waves in bubbly media.
\newblock 35(7):1975--1998, 2018.

\bibitem{bondarenko2016factorization}
O.~Bondarenko.
\newblock {\em The factorization method for conducting transmission
  conditions}.
\newblock PhD thesis, Dissertation, Karlsruhe, Karlsruher Institut f{\"u}r
  Technologie (KIT), 2016, 2016.

\bibitem{Bonnetier2000}
E.~Bonnetier and M.~Vogelius.
\newblock An elliptic regularity result for a composite medium with" touching"
  fibers of circular cross-section.
\newblock {\em SIAM Journal on Mathematical Analysis}, 31(3):651--677, 2000.

\bibitem{budiansky1984high}
B.~Budiansky and G.~Carrier.
\newblock High shear stresses in stiff-fiber composites.
\newblock 1984.

\bibitem{chen2021estimates}
Y.~Chen and H.~Li.
\newblock Estimates and asymptotics for the stress concentration between
  closely spaced stiff c1, $\gamma$ inclusions in linear elasticity.
\newblock {\em Journal of Functional Analysis}, 281(2):109038, 2021.

\bibitem{Colton2019}
D.~Colton and R.~Kress.
\newblock {\em Inverse Acoustic and Electromagnetic Scattering Theory},
  volume~93.
\newblock Springer Nature, 2019.

\bibitem{Deng2022a}
Y.~Deng, X.~Fang, and H.~Liu.
\newblock Gradient estimates for electric fields with multiscale inclusions in
  the quasi-static regime.
\newblock {\em Multiscale Modeling \textup{\&} Simulation}, 20(2):641--656,
  2022.

\bibitem{Deng2022}
Y.~Deng, Y.~Hu, and W.~Tang.
\newblock Optimal estimate of field concentration between multiscale
  nearly-touching inclusions for 3-d helmholtz system.
\newblock {\em Communications on Analysis and Computation}, 1(1):56--71, 2023.

\bibitem{Dong2021}
H.~Dong, Y.~Li, and Z.~Yang.
\newblock Optimal gradient estimates of solutions to the insulated conductivity
  problem in dimension greater than two.
\newblock {\em arXiv preprint arXiv:2110.11313}, 2021.

\bibitem{Dong2024}
H.~Dong, Y.~Li, and Z.~Yang.
\newblock Gradient estimates for the insulated conductivity problem: the
  non-umbilical case.
\newblock {\em Journal de Math{\'e}matiques Pures et Appliqu{\'e}es},
  189:103587, 2024.

\bibitem{kang2014characterization}
H.~Kang, M.~Lim, and K.~Yun.
\newblock Characterization of the electric field concentration between two
  adjacent spherical perfect conductors.
\newblock {\em SIAM Journal on Applied Mathematics}, 74(1):125--146, 2014.

\bibitem{kang2019quantitative}
H.~Kang and S.~Yu.
\newblock Quantitative characterization of stress concentration in the presence
  of closely spaced hard inclusions in two-dimensional linear elasticity.
\newblock {\em Archive for Rational Mechanics and Analysis}, 232:121--196,
  2019.

\bibitem{kress2013linear}
R.~Kress.
\newblock {\em Linear Integral Equations}.
\newblock Applied Mathematical Sciences. Springer New York, 2013.

\bibitem{li2022minnaert}
H.~Li, H.~Liu, and J.~Zou.
\newblock Minnaert resonances for bubbles in soft elastic materials.
\newblock {\em SIAM Journal on Applied Mathematics}, 82(1):119--141, 2022.

\bibitem{li2023interaction}
H.~Li and Y.~Zhao.
\newblock The interaction between two close-to-touching convex acoustic
  subwavelength resonators.
\newblock {\em Multiscale Modeling and Simulation}, 21(3):804--826, 2023.

\bibitem{Li2003}
Y.~Li and L.~Nirenberg.
\newblock Estimates for elliptic systems from composite material.
\newblock {\em Communications on Pure and Applied Mathematics: A Journal Issued
  by the Courant Institute of Mathematical Sciences}, 56(7):892--925, 2003.

\bibitem{li2023gradient}
Y.~Li and Z.~Yang.
\newblock Gradient estimates of solutions to the insulated conductivity problem
  in dimension greater than two.
\newblock {\em Mathematische Annalen}, 385(3-4):1775--1796, 2023.

\bibitem{Liu2011}
H.~Liu, Z.~Shang, H.~Sun, and J.~Zou.
\newblock Singular perturbation of reduced wave equation and scattering from an
  embedded obstacle.
\newblock {\em Journal of Dynamics and Differential Equations}, 24, 11 2011.

\bibitem{Markenscoff1992}
X.~Markenscoff and J.~Dundurs.
\newblock Amplification of stresses in thin ligaments.
\newblock {\em International journal of solids and structures},
  29(14-15):1883--1888, 1992.

\bibitem{pierce2019acoustics}
A.~D. Pierce.
\newblock {\em Acoustics: an introduction to its physical principles and
  applications}.
\newblock Springer, 2019.

\end{thebibliography}

\end{document}